\pdfoutput=1
\documentclass[10pt,reqno]{amsart}
\usepackage{bbm}
\usepackage{amsmath,amsthm,amsfonts,amssymb,amscd,mathrsfs}
\usepackage{amssymb,amsfonts,amsthm, color, algorithm}
\usepackage{bm, textcomp}
\usepackage{psfrag,graphicx}
\usepackage{epic,eepic}
\usepackage{color}
\usepackage{ebezier}
\usepackage{epsfig}
\usepackage{epstopdf}
\usepackage{abstract}

\theoremstyle{plain}
\newtheorem{Main}{Theorem}

\newtheorem{Theorem}{Theorem}
\newtheorem{Prop}[Theorem]{Proposition}
\newtheorem{Lemma}[Theorem]{Lemma}

\newtheorem*{Remark}{Remark}

\theoremstyle{remark}

\newtheorem{Definition}[Theorem]{\bf Definition}

\newcommand{\orb}{\operatorname{orb}}

\def\ln{\operatorname{ln}}

\def\top{\operatorname{top}}

\def\max{\operatorname{max}}
\def\min{\operatorname{min}}

\def\exp{\operatorname{exp}}

\def\vep{\varepsilon}

\begin{document}

\title[\,]{Quasi-shadowing  property for nonuniformly partially hyperbolic systems }
\author{Gang Liao and Xuetong Zu }



\thanks{School of Mathematical Sciences, Center for Dynamical Systems and Differential Equations, Soochow University,
	Suzhou 215006,  China;  This research  was partially supported by  National Key R\&D Program of China  (2022YFA1005802),   NSFC (12071328,   12122113) and    BK20211595}

\email{lg@suda.edu.cn}
\email{zuxuetong@163.com}

\maketitle
\begin{abstract}

	In this paper,  we    establish  a new quasi-shadowing property for any  nonuiformly partially hyperbolic  set of  a $C^{1+\alpha}$  diffeomorphism,  which  is  adaptive to the  movement of the pseudo-orbit.  Moreover, the quasi-specification property and quasi-closing property  are  also investigated.   As an  application of quasi-closing property, we  extend Katok's  reslut on the growth of periodoc orbits  for hyperbolic ergodic measure  to any ergodic measure:  the number  of quasi-periodic  points grows exponentially  at least  the metric entropy.

\end{abstract}

\section{Introduction}

The mathematical models used in practical applications are often simplified compared to reality, yet still need to be effective,  thus such  properties that remain unchanged or can be approxiamted  to some extent under small perturbations attract  people usually.   In this paper, we are concerned with the approximation mechanism of pseudo-orbits by real orbits, i.e., the shadowing property of dynamical system.    The shadowing property is a poweful tool in the structrual anlysis  of differentiable dynamical system, see  Liao \cite{Lia},  Ma\~n\'e \cite{Man},  Wen \cite{Wen}  and  Gan  \cite{Gan}  for the investigation of stability,  and Bowen \cite{MR0262459},  Katok \cite{MR0573822} and Sarig \cite{Sar} for the investigation of complexity. 
\\

\noindent {\bf (1) Shadowing for uniformly hyperbolic systems.}  Throughout,  let $M$ be a compact  boundaryless  Riemannian manifold and,    $\|\cdot\|$ and   $d$ be the metric and norm induced by the  Riemannian metric.  Let $f$ be a diffeomorphism on $M$ and    $\Lambda$ be a   uniformly hyperbolic set:  there exist  an $Df$-invariant splitting $T_{\Lambda} M=E^s\oplus E^u$,   constants $C\ge 1$ and $0<\lambda <1$ such that 

\begin{itemize}
	\item[(i)]$
	\|D_xf^n(v)\|\le C\lambda^n\|v\|,\quad \forall\,x\in \Lambda,\,v\in E^s(x);$\\
	\item[(ii)]$	\|D_xf^{-n}(v)\|\le C\lambda^n\|v\|,\quad \forall\,x\in \Lambda,\,v\in E^u(x).$
\end{itemize}
For $ \delta>0 $, we refer to $ \{x_n\}_{n\in\mathbb{Z}} $ as a $ \delta $ pseudo-orbit if $ d(f(x_n),x_{n+1})<\delta,\;n\in \mathbb{Z} $.   It holds the following  shadowing property  (Anosov \cite{Ano}, Bowen \cite{Bow} and Sinai \cite{Sin}):   
For any $\vep>0$, there exists $ \delta>0 $ such that for any $ \delta $ pseudo-orbit $ \{x_n\}_{n\in\mathbb{Z}}\subset\Lambda$, there exists a point $ x $ such that
\begin{equation*}
d(x_n,f^n(x))<\varepsilon,\;n\in\mathbb{Z}.
\end{equation*}
Besides,   there exists $\vep_0>0$ such that if $\vep<\vep_0$, the shadowing point $x$ is unique.    By considering periodic  pseudo-orbit,  there is  closing property:
for any $ \varepsilon>0 $, there exists $ \beta>0 $  such that for any $ x $ with  $	d(x,f^p(x))<\beta$ for some $p\in \mathbb{Z}^+$, 
there exists a point $z$ with $f^p(z)=z$ such that
$$
d(x,z)<\varepsilon.$$

Gan \cite{MR1897871} proved a generalized shadowing lemma for a closed invariant set with a continuous invariant splitting.
Wen et al. \cite{MR2467027} proved that the chain component that contains a hyperbolic periodic point is hyperbolic if it is $ C^1 $-stably shadowable.

\noindent {\bf (2) Shadowing for nonuniformly hyperbolic systems.}     For nonuniformly hyperbolic systems,  the hyperbolicity is not uniform,  thus the scale of  shadowing property needs to be adjusted to the level of hyperbolicity. 

Let  $f$ be a $C^{1+\alpha}$ diffeomorphism on $M$.  	Given positive numbers $ \lambda,\;\mu $ and $\varepsilon $ with $ \lambda,\mu \gg \varepsilon $ and a positive integer $ k  $,  define $  \Lambda_k=\Lambda_k(\lambda,\mu,\varepsilon)  $ to be all points $ x\in M $ for which there is an $Df$-invariant decomposition $ T_{\orb(x)}M=E^s\oplus E^u $ satisfying

\begin{itemize}
	\item[(i)]$ \| Df^n(v)\|\le e^{\varepsilon k}e^{-(\lambda-\varepsilon)n}e^{\varepsilon|m|}\|v\|,\,\, \forall n\ge 0, m\in\mathbbm{Z}, \,v\in E^s({f^m(x)})$;\\
	\item[(ii)] $ \| Df^{-n}(v)\|\le e^{\varepsilon k}e^{-(\mu-\varepsilon)n}e^{\varepsilon|m|}\|v\|,\,\,\forall n\ge 0, m\in\mathbbm{Z},\,v\in  E^u(f^m(x))$.
\end{itemize}
We refer to $ \Lambda_{k} $ as a  hyperbolic block or   Pesin block \cite{Pes}.  It is veried  that $ \Lambda_k\subseteq\Lambda_{k+1} $ from the definition.    The hyperbolicity of $\Lambda_k$ is weaker as $k$ becomes larger.     Denote $\Lambda=\bigcup\limits_{k\ge 1}\Lambda_k(\lambda,\mu,\varepsilon)  $ and we call $ \Lambda $ a  nonuniformly hyperbolic set. 	One may see that if there is some ergodic hyperbolic measure, then we can get a  nonuniformly hyperbolic set of full measure by taking $\lambda$ as the absolute value of the maximal negative Lyapunov exponent and $\mu$ as the minimun  positive  Lyapunov exponent.

For a sequence of numbers  $ \{\delta_{k} \}_{k\in\mathbb{Z^+}}$, if there exist a sequence of positive integers $ \{s_n\}_{n\in\mathbb{Z}} $ and a sequence of points  $ \{x_n\}_{n\in\mathbb{Z}}\subset \Lambda $ such that

\begin{itemize}
	\item[(a)]$ x_n\in\Lambda_{s_n} $, $ \forall n\in\mathbb{Z} $;\\
	\item[(b)] $ |s_n-s_{n+1}|\le 1 $, $ \forall n\in\mathbb{Z} $;\\
	\item[(c)]$d(f(x_n),x_{n+1})< \delta_{s_n} , \forall n\in\mathbb{Z} $,
\end{itemize}
we refer to $ \{x_n\}_{n\in\mathbb{Z}}\subset \Lambda $ as a $ \{\delta_{k} \}_{k\in\mathbb{Z^+}}$ pseudo-orbit of $ f $.

Katok \cite{MR0573822} proved the following  shadowing property:  for any  $ \eta>0 $, there exists a sequence of positive numbers  $ \{\delta_k\}_{k\in\mathbb{Z^+}} $ such that for any $ \{\delta_{k} \}_{k\in\mathbb{Z^+}}$ pseudo-orbit $ \{x_n\}_{n\in\mathbb{Z}}\subset \Lambda $, there exists a point  $ x $ satisfying:
\begin{equation*}
d(x_n,f^n(x))<\eta \epsilon_{s_n},
\end{equation*}
where $ \epsilon_{s_n}=\epsilon_0e^{-s_n\cdot \frac{\varepsilon}{\alpha}} $ and $ \epsilon_0 $ is a number  independent of $ s_n $.
Besides,   if $\eta<1$, the shadowing point $x$ is unique. 
There is also  closing property   according to the scale of hyperbolic block $\Lambda_k$.   For each $k\in \mathbb{Z}^+$ and  any $ \eta>0 $,  
there exists a positive number $ \beta=\beta(k,\eta) $  such that  for any  point $ x\in\Lambda_k $ with  $ d(x, f^p(x))<\beta $, there exists a point $z$ with $f^p(z)=z$ 	such that:
$$ d(x,z)<\eta.$$

\noindent {\bf (3) Quasi-shadowing for partially  hyperbolic systems.} 

Comparing to hyperbolic systems, the partially hyperbolic systems  has additional central direction.  The shadowing property does not exist for all partially hyperbolic systems any more \cite{MR3007732}. Consequently, the research turns to  the concept of quasi-shadowing property and  sequences of shadowing points (not orbits in general) are allowed to make minor adjustments in the central direction.  

Let  $f$ be a diffeomorphism on $M$ and    $\Lambda$ be a   (uniformly) partially  hyperbolic set:  there exist  an $Df$-invariant splitting $T_{\Lambda} M=E^s\oplus E^c\oplus E^u$,   constants $C\ge 1$, $0<\lambda<1<\mu$ and $0<\lambda<\lambda'\le \mu'< \mu <1$ such that 

\begin{itemize}
	\item[(i)]$
	\|D_xf^n(v)|\le C\lambda^n\|v\|,\quad \forall\,x\in \Lambda,\,v\in E^s(x)$;\\
	\item[(ii)]	$\frac{1}{C}(\lambda')^n\|v\|\le \|D_xf^n(v)\|\le C(\mu')^n\|v\|,\quad \forall\,x\in \Lambda,\,v\in E^c(x)$;\\
	\item[(iii)]	$	\|D_xf^{-n}(v)|\le C\mu^{-n}\|v\|,\quad \forall\,x\in \Lambda,\,v\in E^u(x)$.
\end{itemize}
Take  a constant $\rho_0 > 0$ such that for any $x \in M$ the standard
exponential map $\exp_x: \{v \in  T_x M:  \|v\| < \rho_0\}\to M$  is a  $C^{\infty}$ diffeomorphism to its 
image.  Hu, Zhou and Zhu  \cite{MR3316919} proved that  on partially hyperbolic set, $f$  has the following quasi-shadowing property:  for any $\vep\in (0,\rho_0)$, there exists $\delta>0$  such
that for any $\delta$-pseudo orbit $\{x_n \}_{n\in \mathbb{Z}}$, there exist a sequence of points $\{y_n \}_{n\in\mathbb{Z}}$ and a
sequence of vectors $\{u_k\in E^c(x_k)\}_{n\in\mathbb{Z}}$  such that
$$d(x_n , y_n ) < \vep,\,\forall\,n\in\mathbb{Z},$$ 
where
$$y_n = \exp_{x_n}(u_n+\exp_{x_n}^{-1}(f(y_{n-1}))).$$
Besides,  $\{y_n \}_{n\in\mathbb{Z}}$  and $\{u_n \}_{n\in\mathbb{Z}}$ can be chosen uniquely so as to satisfy
$$y_n\in \exp_{x_n}(E^s(x_n)+E^u(x_n)).$$ 
Here $u_k\in E^c(x_k)$ is   taken  into account  as a motion along the center direction.

if for any pseudo orbit, there is a sequence of points $ \{y_k\}_{k\in\mathbb{Z}} $ tracing it, in which $ y_{k+1} $ is obtained from $ f(x_k) $ by a motion  along the center direction.
In \cite{LZ}, Li et al. studied the quasi-shadowing property for  partially hyperbolic flows.
In \cite{LZ1},  Li et al. proved the quasi-shadowing property for quasi-partially hyperbolic pseudo-orbit made by quasi-partially hyperbolic strings.\\

\noindent {\bf (4)Quasi-shadowing  for  nonuniformly  partially  hyperbolic systems.} 

Nonuniformly partially hyperbolic system is a generalization of  nonuniformly hyperbolic system  by taking into account the center boundle and also a generalization of uniformly partially hyperbolic system by modifying the uniformity.   Nonuniformly partially hyperbolic sets  always exist, as we will see.

Let  $ f$  be a $ C^{1+\alpha} $ diffeomorphism on $M$. 

\begin{Definition}\label{def1}
	Given numbers $ \lambda, \mu >0$,   $-\lambda<- \lambda'<\mu'<\mu $,  $\varepsilon >0$ with $\vep\ll \min\{ -\lambda, \mu, |\lambda-\lambda'|, |\mu'+\lambda'|, |\mu-\mu'| \}$,  and a positive integer $ k  $, we define $  \Lambda_k=\Lambda_k(\lambda,\mu,\lambda',\mu',\varepsilon)  $ to be all points $ x\in M $ for which there is a $Df$-invariant decomposition $ T_{\orb(x)}M=E^s\oplus E^c\oplus E^u $ along the orbit of $x$  satisfying
	
	\begin{itemize}
		\item[(i)]$ \| Df^n(v)\|\le e^{\varepsilon k}e^{-(\lambda-\varepsilon)n}e^{\varepsilon|m|}\|v\|,\,\, \forall n\ge 0, m\in\mathbbm{Z},\,\,v\in E^s({f^m(x)})$;\\
		\item[(ii)] $ \| Df^n(v)\|\le e^{\varepsilon k}e^{(\mu'+\varepsilon)n}e^{\varepsilon|m|}\|v\|$,\\
		$ \| Df^{-n}(v)\|\le e^{\varepsilon k}e^{(\lambda'+\varepsilon)n}e^{\varepsilon|m|}\|v\|,
		\,\,\forall n\ge 0,m\in\mathbbm{Z},\,\,v\in E^c({f^m(x)})$;\\
		\item[(iii)] $ \| Df^{-n}(v)\|\le e^{\varepsilon k}e^{-(\mu-\varepsilon)n}e^{\varepsilon|m|}\|v\|,\,\,\forall n\ge 0, m\in\mathbbm{Z},\,\, v\in E^u(f^m(x)) $.
	\end{itemize}
	
	We refer to $ \Lambda_{k} $ as a partially hyperbolic block or   Pesin block in nonuniformly partially hyperbolic sense.  Denote by $ \kappa=\kappa(x) $ the  minimum value of the $ k $ that satisfy the aforementioned conditions for $x$ and we refer to it as the index of  partial hyperbolicity.  Denote $\Lambda=\Lambda(\lambda,\mu,\lambda',\mu',\varepsilon)=\bigcup\limits_{k\ge 1}\Lambda_k(\lambda,\mu,\lambda',\mu',\varepsilon)  $ and we call $ \Lambda $ a nonuniformly partially hyperbolic set.
\end{Definition}

\begin{Remark} Nonuniformly  partially  hyperbolic set  always exist.   Given an ergodic measure which always exists,  we can get a  nonuniformly partially hyperbolic set of full measure by taking $\lambda$ as the absolute value of the maximal negative Lyapunov exponent,  $\mu$ as the minimun  positive  Lyapunov exponent and $\lambda'<\lambda$, $\mu'<\mu$.  Here $E^s$ or $E^u$ may be trivial if $\mu$ has no  negative Lyapunov exponent or positive   Lyapunov exponent. 
\end{Remark}

In this paper, we shall study  the quasi-shadowing property for  nonuniformly partially hyperbolic systems.     For any $x\in M$ and a positive integer $n$, denote the orbit segment $$[x, f^n(x)] :=\{x, f(x),\cdots, f^n(x)\}.$$   On  the orbit segment, there is no movement and it is real orbit.   We want to show  that   if it is an orbit segment in the pseudo-orbit, it is shadowed by an orbit segment in the shadowing orbit.

There exists $ \{\delta_{k} \}_{k\in\mathbb{Z^+}}$ such that for a   $ \{\delta_{k} \}_{k\in\mathbb{Z^+}}$ pseudo-orbit $\{z_i\}$,	 denoting $\tau(z_i)=s_i$, there exist fake  stable foliation $ \mathcal{W}^s $,  fake center foliation $ \mathcal{W}^{c} $ and fake  unstable foliation $ \mathcal{W}^u $  inside  $ B(z_i, \epsilon_0e^{-\tau(z_i)\frac{\varepsilon}{\alpha} }) $,  where $\epsilon_0$ is  independent of  $\tau(z_i)$.  We refer to Section \ref{Uniform partial hyperbolicity} for more details on the fake  foliations.

We have the following quasi-shadowing property  adaptive to the  movement of the pseudo-orbit, i.e., the shadowing orbit  has movement  only if  the pseudo-orbit has movement. 

\begin{Main}\label{thA}
	Let $ f$  be a  $ C^{1+\alpha} $ diffeomorphism on  $M$ and   $ \Lambda=\bigcup\limits_{k\ge 1}\Lambda_k(\lambda,\mu,\lambda',\mu',\varepsilon)  $ be    a nonuniformly partially hyperbolic set. Then $ f $ has the quasi-shadowing property for $\Lambda$ in the following sense:  for any $ \eta>0 $ there exists  a  sequence of positive numbers $ \{\delta_k\}_{k\in\mathbb{Z}^+}  $  such that for any $  \{\delta_{k} \}_{k\in\mathbb{Z^+}}$ pseudo-orbit   $ \{z_i\}= \{\cdots, [x_{-n}, f^{a_{-n}-1}(x_{-n})], $ $  \cdots, [x_0, f^{a_0-1}(x_0)], \cdots, [x_{n},  f^{a_n-1}(x_{n})], \cdots \}  \subset  \Lambda $,  there exists a sequence of points $ \{y_n\}_{n\in\mathbb{Z}} $ satisfying:
	
	\begin{itemize}
		\item[(i)]$ f^{a_n}(y_n)\in  \mathcal{W}^c(y_{n+1}),\,  n\in \mathbb{Z} $, where $ \mathcal{W}^c $ is the fake center foliation of $ f $;\\
		\item[(ii)]$ d(f^{i}(x_{n}),f^{i}(y_{n})) < \eta\epsilon_{\tau(x_n)},\;0\le i\le a_n,\,\, n\in\mathbb{Z}. $
	\end{itemize}	
\end{Main}

\vspace*{10pt}
\begin{figure}[H]
	\begin{center}
		\includegraphics[width=0.85\textwidth]{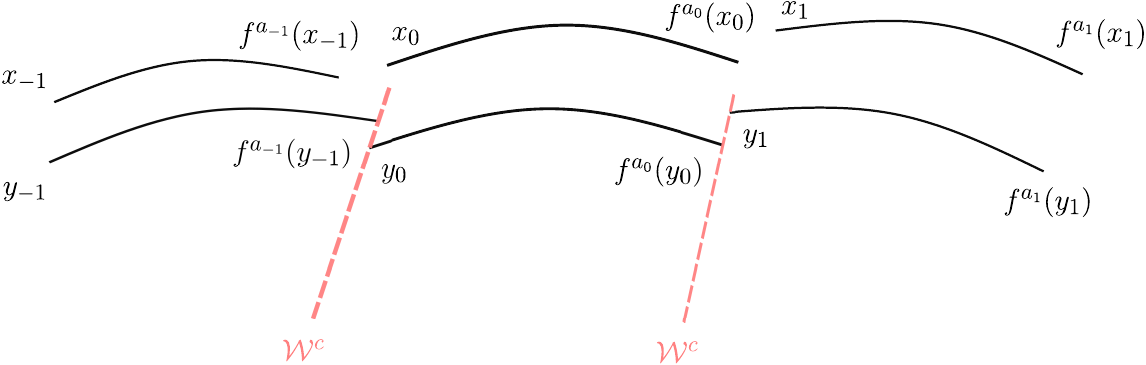}
	\end{center}
	\caption{Quasi-shadowing}
\end{figure}
\vspace*{10pt}

In the category of invariant sets, periodic orbits are important invariant sets, and in the context of invariant measures, periodic measures are significant measures. Regarding periodic points, many scholars have already derived some results. Sigmund \cite{Sig} proved that in nuniformly hyperbolic systems, periodic measures are dense in the set of invariant measure. Hirayama \cite{MR1974422}  proved that in nonuniformly hyperbolic systems, periodic measures are dense in the set of invariant measures supported on a total measure set with respect to a hyperbolic mixing measure.  Liang et al \cite{MR2457408} replaced the assumption of hyperbolic mixing measure by a more natural and weaker assumption of hyperbolic ergodic measure and generalized Hirayama's result. Indeed, many properties of hyperbolic measures can be well
approximated by periodic measures.

For $ n,p >0 $, denote by $ P'_{n,p}(f)=\{y:\; $ there exist a sequence of number $ \{n_i\}_{i=0}^{p-1} $ satisfying $ \sum\limits_{i=0}^{p-1}n_i=n $ and  a sequence of point $ \{z_i\}_{i=1}^{p-1} $ such that $ z_1\in\mathcal{W}^c(f^{n_0}(x)) $, $ z_{i+1}\in\mathcal{W}^c(f^{n_i}(z_i)),i=1,\cdots , p-2$, $ x\in\mathcal{W}^c(f^{n_{p-1}}(z_{p-1}))\} $, where $ \mathcal{W}^c $ is the fake center  foliation of $ f $.  We refer to $ P'_{n,p}(f) $ as the $ p $ quasi-periodic points of $ f $ with period $ n $.

We have the following quasi-specification property concerning periodic shadowing.

\begin{Main}\label{thD}
	Let $ f:M\to M $ be a  $ C^{1+\alpha} $ diffeomorphism on a boundless compact smooth Riemannian manifold and there are  a nonuniformly partially hyperbolic set $ \Lambda=\bigcup\limits_{k\ge 1}\Lambda_k(\lambda,\mu,\lambda',\mu',\varepsilon)  $. Let  $ m $ be a Borel probability $ f$-invariant  ergodic measure and denote $ \tilde{\Lambda}_k$ as the support of $\mu$ restricted on  $\Lambda_k$ and $ \tilde{\Lambda}=\bigcup\limits_{k\ge1}\tilde{\Lambda}_k $. Then $ f $ has the quasi-specification property for $\tilde{\Lambda} $ in the following sense:  for any $ \eta>0 $, there exists a set of  positive integers  $\{M_{k,t}=M_{k,t}(\eta)\}_{k,t\in \mathbb{Z}^+}$ such that for   any sequence of orbit segment $[x_1, f^{a_1}(x_1)], \cdots, [x_l, f^{a_l}(x_l)]$  with $x_i, f^{a_i}(x_i)\in \tilde{\Lambda}_{k_i}$ for  $1\le i\le l$, there exist  two sequences of points  $ \{y_i\}_{i=i}^{l} $ and  $ \{\hat{y}_i\}_{i=1}^{l} $, and a sequence of positive integers $ \{X_{1, 2},\cdots, X_{l-1, l}, X_{l, 1}\} $ with $X_{i, i+1}\le M_{k_i, k_{i+1}}$, $1\le i\le l-1$, and $X_{l,1}\le M_{k_l, k_1}$  satisfying:
	
	\begin{itemize}
		
		\item[(i)]
		$f^{a_i}(y_i) \in\mathcal{W}^c( \hat{y}_i) $, 
		$ f^{X_{i, i+1}}(\hat{y}_i)  \in\mathcal{W}^c (y_{i+1}),\; i=1, \cdots, l-1 $,  $f^{a_l}(y_l) \in\mathcal{W}^c( \hat{y}_l) $,   and  $f^{X_{l, 1}}(\hat{y}_l)  \in\mathcal{W}^c (y_{1}) $, where $ \mathcal{W}^c $ is the fake center  foliation of $ f $, i.e.,    
		$ y_1 $ is a $ 2l $ quasi-periodic point of $ f $;\\

		\item[(ii)] $ d(f^{j}(y_i),f^j(x_i))\le \eta  \epsilon_{k_i},\;\; 0\le j\le a_i, \,\,i=1,\cdots,l. $
	\end{itemize}
	
\end{Main}

\vspace*{10pt}
\begin{figure}[H]
	\begin{center}
		\includegraphics[width=0.85\textwidth]{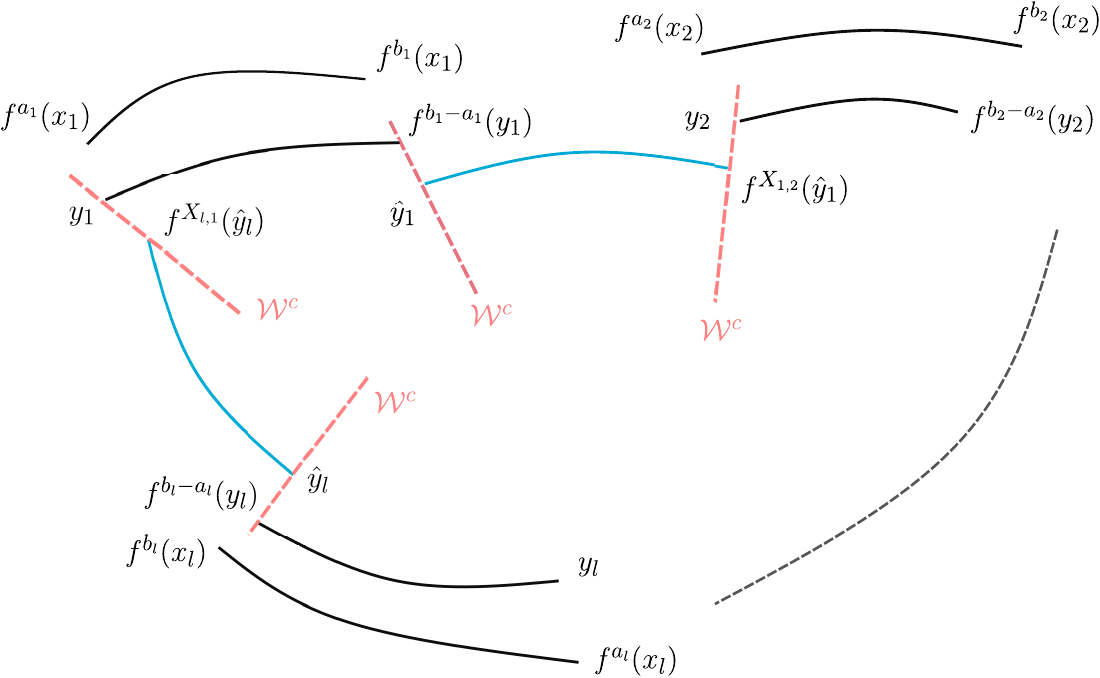}
	\end{center}
	\caption{Quasi-specification}
\end{figure}
\vspace*{10pt}

The above quasi-specification property can be derived from the quasi-shadowing property as follows. 

For any $ \eta>0 $ and $k\ge 1 $, take  $\delta_k$  as in Theorem  \ref{thA}.   Let $ \{M_{k,l}=M_{k,l}(\eta)\}_{k,l\in \mathbb{Z}^+}$ be as the constrction in \cite{LST}  (Page 567). 
For  each $1\le i\le l$,   there exist $z_i\in B(f^{a_i}(x_i), \delta_{k_i})\cap \Lambda_{k_i}$ and a sequence of positive integers $ \{X_{1, 2},\cdots, X_{l-1, l}, X_{l, 1}\} $ with  $X_{i,i+1}\le   M_{k_i, k_{i+1}}$,   $1\le i\le l-1$,  $X_{l,1}\le  M_{k_l, k_1}$  satisfying \begin{eqnarray*} &&f^{X_{i,i+1}}(z_i)\in B(x_{i+1}, \delta_{k_{i+1}})\cap \Lambda_{k_{i+1}},\quad 1\le i\le l-1,\\[2mm] && f^{X_{l,1}}(z_l)\in B(x_1, \delta_{k_1})\cap \Lambda_{k_1},\end{eqnarray*}
and
\begin{eqnarray*}
	&\big{\{}& x_1,\cdots,f^{a_1-1}(x_1), z_1,\cdots,f^{X_{1, 2}-1}(z_1),
	x_2,\cdots,f^{a_2-1}(x_2),
	z_2,\cdots,f^{X_{2,3}-1}(z_2),
	\\[2mm]
	&& \cdots, x_l,\cdots,f^{a_l-1}(x_l), z_l,\cdots,f^{X_{l, 1}-1}(z_l)
	\big{\}}^{\infty}
\end{eqnarray*}
is a   periodic  $\{\delta_k\}_{k\in\mathbb{Z}^+} $-pseudo-orbit.   Theorem \ref{thD} is obtained by applying Theorem  \ref{thA}.

The statement of the quasi-closing property is as follows:

\begin{Main}\label{thB}
	Let $ f:M\to M $ be a  $ C^{1+\alpha} $ diffeomorphism on a boundless compact smooth Riemannian manifold $M$ and there is  a nonuniformly partially hyperbolic set $ \Lambda=\bigcup\limits_{k\ge 1}\Lambda_k(\lambda,\mu,\lambda',\mu',\varepsilon)  $. Then $ f $ has the quasi-closing property for $\Lambda$ in the following sense:  for any $ \eta>0 $ and $k\ge 1 $ there exists  $ \beta=\beta(k,\eta)>0 $  such that  for any point $ x\in \Lambda_k $ if for some  $ p\in\mathbb{N} $ one has 
	\begin{equation*}
	f^p(x)\in \Lambda_k
	\end{equation*}
	and
	\begin{equation*}
	d(x,f^p(x))<\beta,
	\end{equation*}
	then there exists $y\in M$  such that:
	\begin{itemize}
		\item[(i)]$ f^{p}(y)\in\mathcal{W}^c(y)$, where $ \mathcal{W}^c $ is the fake center  foliation of $ f $;
		\item[(ii)]$ d(f^i(x), f^i(y))<\eta\epsilon_k ,\; i=0,1,\cdots,p-1  $.
	\end{itemize}
	
\end{Main}

For any $ \eta>0 $ and $k\ge 1 $, take  $\delta_k$  as in Theorem  \ref{thA} and let $\beta(k,\eta)=\delta_k$. Then Theorem \ref{thB} is obtained by taking  periodic  pseudo-orbit 
$$ \big{\{} x, f(x), \cdots, f^{p-1}(x) \big{\}}^{\infty}. $$

\vspace*{10pt}
\begin{figure}[H]
	\begin{center}
		\includegraphics[width=0.45\textwidth]{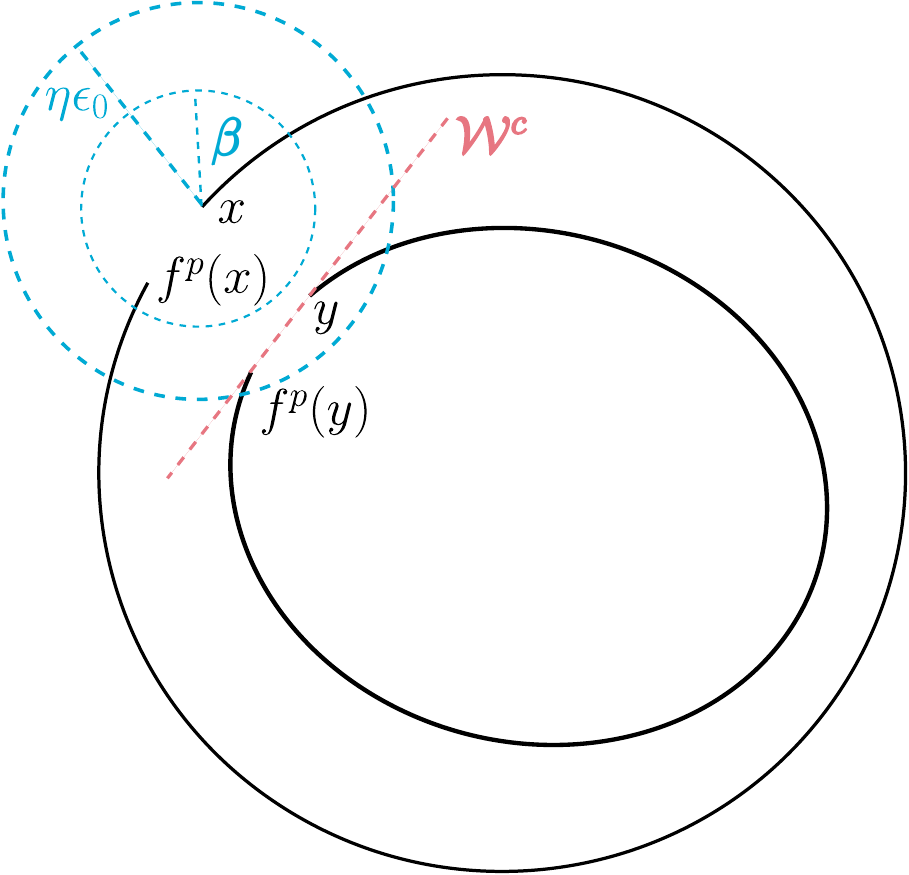}
	\end{center}
	\caption{Quasi-closing}
\end{figure}
\vspace*{10pt}

Entropy is a concept widely applied across various fields including physics, information theory, statistics, and mathematics. Initially introduced by physicist  Clausius in the mid-19th century, it describes the degree of disorder in thermodynamic systems. In different disciplines, entropy has various meanings and applications. In the theory of dynamical systems, entropy serves as a measure of the complexity of the system's dynamics. The larger the entropy of a dynamical system is, the less predictable the system is. Topological entropy and metric entropy are two distinct types of entropy. The concept of metric entropy was introduced by Kolmogorov in 1958 and is closely related to entropy in information theory, while topological entropy was introduced by Adler, Konheim, and McAndrew in 1965. In certain cases, metric entropy and topological entropy can be converted or complement each other, providing us with different perspectives to understand and analyze the complexity of dynamical systems. Interestingly, topological entropy and metric entropy are related to the number of periodic points.

Denote by $ \#P_n(f) $ the number of periodic points with period $ n $. Katok \cite{MR0573822} uesd the closing property proved an estimation of asympotic growth of the numbers of $ \#P_n(f) $ by the metric entropy:
\begin{equation*}
\varlimsup_{n\to \infty}\frac{\ln \#P_n(f)}{n}\ge h_m(f),
\end{equation*}
where $ m $ is a Borel probability $ f $-invariant measure with non-zero Lyapunov exponents and  $ \# A $  is  the cardinality of a set $ A$.   There have been someone  who have demonstrated the relationship between the topological entropy and the exponential growth rate of periodic points. For any uniformly hyperbolic system, Bowen \cite{MR0262459} proved that the asymptotical exponential growth of the number of periodic points was determined by the topological entropy $h_{\top}$:
\begin{equation*}
\varlimsup_{n\to \infty}\frac{\ln \#P_n(f)}{n}= h_{\top} (f).
\end{equation*}

As the application of Theorem \ref{thB}, we extend Katok's result of hyperbolic ergodic measure to any ergodic measure. 

Recall that  $$P'_{n,1}(f)=\{x: f^n(x)\in \mathcal{W}^c(x)\}.$$
and  let 
$$P'_{n,1}(f,\vep)=\max\# E$$
with $E\subseteq P'_{n,1}(f)$ satisfying:  $$\,\forall\,x\neq y\in E:\,\,\exists i\in[0,n),  d(f^i(x), f^i(y))>\vep.$$

\begin{Main}\label{thC}
	Let  $ f:M\to M $ be a nonuniformly partially hyperbolic $ C^{1+\alpha} $ diffeomorphism on a boundaryless compact smooth Riemannian manifold and there is  a nonuniformly partially hyperbolic set $ \Lambda=\bigcup\limits_{k\ge 1}\Lambda_k(\lambda,\mu,\lambda',\mu',\varepsilon)  $. Let  $ m $ be an ergodic  $ f$-invariant  measure with full measure for $\Lambda$.  Then	
	
	\begin{equation*}
	\lim_{\vep\to 0}	\varlimsup_{n\to \infty}\frac{\ln P'_{n,1}(f,\vep)}{n}\ge h_m(f).
	\end{equation*}
\end{Main}

\section{Uniform partial hyperbolicity}\label{Uniform partial hyperbolicity}
Let $ M $ be a boundaryless compact smooth Riemannian manifold,  and $ f:M \to M $  be   a  $ C^{1+\alpha} $ diffeomorphism which has  a nonuniformly partially hyperbolic set $ \Lambda=\bigcup\limits_{k\ge 1}\Lambda_k(\lambda,\mu,\lambda',\mu',\varepsilon)  $.

By the definition of nonuniformly partially hyperbolic diffeomorphism,  the time to admit contracting  or expanding in  $ E^s_x $ or  $E^u_x $  is related to the position of $ x $.  Replacing $ \parallel \cdot \parallel $ with another metric $ \parallel \cdot \parallel'' $, which is equivalent to  $ \parallel \cdot \parallel $ on any given $ \Lambda_{k} $, we can get  uniformly   contracting or   expanding on  $ \Lambda $.  Further details are provided below.

Let 
\begin{equation*}
\begin{cases}
\lambda_1=\lambda-2\varepsilon,\\
\mu_1=\mu-2\varepsilon,\\
\lambda'_1=\lambda'+2\varepsilon,\\
\mu'_1=\mu'+2\varepsilon.\\
\end{cases}
\end{equation*}
For $ x\in\Lambda $, define
\begin{eqnarray*}
	\parallel v_{s} \parallel' _{s}&=&\sum_{n=0}^{\infty}e^{\lambda_1n}\parallel D_xf^n(v_s)\parallel, \,\, v_s\in E_x^s,\\[2mm]
	\parallel v_{c} \parallel' _{c}&=&\sum_{n=0}^{\infty}e^{-\mu'_1n}\parallel D_xf^n(v_c)\parallel+ \sum_{n=1}^{\infty}e^{-\lambda'_1n}\parallel D_xf^{-n}(v_c)\parallel, \,\, v_c\in E_x^c,\\[2mm]
	\parallel v_{u} \parallel' _{u}&=&\sum_{n=0}^{\infty}e^{\mu_1n}\parallel D_xf^{-n}(v_u)\parallel, \,\, v_u\in E_x^u.
\end{eqnarray*}	
Then  let 
\begin{equation*}
\begin{aligned}
&\parallel \cdot \parallel': \, T_\Lambda M \to \mathbb{R}\\[2mm]
&\parallel v \parallel'_x=\max\{\parallel v_{s} \parallel' _{s},\parallel v_{c} \parallel' _{c},\parallel v_{u} \parallel' _{u}\},
\end{aligned}
\end{equation*}
where $ v=v_s+v_c+v_u \in T_xM $ with $ v_i\in E_x^i,\;i=s,c,u $.

With respect to $ x\in\Lambda_k $, it holds that 
\begin{equation*}
\sum_{n=0}^{\infty}e^{\lambda_1n}\parallel D_xf^n(v_s)\parallel\le \sum_{n=0}^{\infty}e^{\lambda_1n}e^{-(\lambda-\varepsilon)}e^{\varepsilon k}\parallel v_s\parallel \le \sum_{n=0}^{\infty}e^{-n\varepsilon}e^{k\varepsilon} \parallel v_s\parallel<\infty.
\end{equation*}
Similarly, we  have $\sum_{n=0}^{\infty}e^{-\mu'_1n}\parallel D_xf^n(v_c)\parallel+ \sum_{n=1}^{\infty}e^{-\lambda'_1n}\parallel D_xf^{-n}(v_c)\parallel<\infty $ and $\displaystyle \sum_{n=0}^{\infty}e^{\mu_1n}\parallel D_xf^{-n}(v_u)\parallel<\infty $.

With the metric $ \parallel \cdot \parallel' $,   $ f:\Lambda \to \Lambda $ behaves  uniformly partially hyperbolic: 

\begin{Prop}\label{metric}
	\begin{eqnarray*} \frac{1}{\|Df^{-1}\|} \|v_s\|'\le  &\parallel Df(v_s)\parallel' &\le e^{-\lambda_2}\parallel v_s \parallel',\;v_s\in E_x^{s} ;\\
		e^{-\lambda'_2}\parallel v_c \parallel' \, \, \le  \,\, & \parallel Df(v_c)\parallel' &\le  e^{\mu''_2}\parallel v_c \parallel',\;v_c\in E_x^{c} ;\\
		\frac{1}{\|Df\|} \|v_u\|' \le 	&\parallel Df^{-1}(v_u)\parallel'&\le  e^{-\mu_2}\parallel v_u \parallel',\;v_u\in E_x^{u}. \end{eqnarray*}
\end{Prop}
\begin{proof}

	\begin{itemize}\item[(1)]$E^s$: \begin{eqnarray*}
			\parallel D_xf(v_{s}) \parallel' 
			&=&\sum_{n=0}^{\infty}e^{\lambda_1n}\parallel D_xf^{n+1}(v_s)\parallel\\
			&=&e^{-\lambda_1}\sum_{n=0}^{\infty}e^{\lambda_1(n+1)}\parallel D_xf^{n+1}(v_s)\parallel\\
			&\le& e^{-\lambda_1}\sum_{n=0}^{\infty}e^{\lambda_1n}\parallel D_xf^{n}(v_s)\parallel\\	
			&\le &
			e^{-\lambda_1}\parallel v_s \parallel'.
		\end{eqnarray*}
		Besides, 
		\begin{eqnarray*}
			\parallel D_xf(v_{s}) \parallel' 
			&=&\sum_{n=0}^{\infty}e^{\lambda_1n}\parallel D_{f^n(x)} f \circ D_xf^{n}(v_s)\parallel\\
			&\ge &\frac{1}{\|Df^{-1}\|} \sum_{n=0}^{\infty} e^{\lambda_1n} \parallel D_xf^{n}(v_s)\parallel\\
			&=&  \frac{1}{\|Df^{-1}\|} \|v_s\|'.
		\end{eqnarray*}
		
		\item[(2)]$E^c$:  
		\begin{eqnarray*}
			\parallel D_xf(v_{c}) \parallel'
			&=&\sum_{n=0}^{\infty}e^{-\mu'_1n}\parallel D_xf^{n+1}(v_c)\parallel+ \sum_{n=1}^{\infty}e^{-\lambda'_1n}\parallel D_xf^{-(n-1)}(v_c)\parallel\\
			&=& e^{\mu'_1}\sum_{n=0}^{\infty}e^{-\mu'_1(n+1)}\parallel D_xf^{n+1}(v_c)\parallel+e^{-\lambda'_1} \sum_{n=1}^{\infty}e^{-\lambda'_1(n-1)}\parallel D_xf^{-(n-1)}(v_c)\parallel,
		\end{eqnarray*}
		which implies 
		$$  e^{-\lambda'_1} \parallel v_c \parallel' \le \parallel D_xf(v_{c}) \parallel'\le e^{\mu'_1} \parallel v_c \parallel'.$$
		\item[(3)]$E^u$:  
		Similarly, we have 
		\begin{equation*} \frac{1}{\|Df\|} \|v_u\|'\le  \parallel D_xf^{-1}(v_{u}) \parallel' \le e^{-\mu_1}\parallel v_u \parallel'.\end{equation*}.
	\end{itemize}
\end{proof}
If $ \parallel v \parallel'=\parallel v_s \parallel'_s $, then 
\begin{eqnarray*}
	\parallel v \parallel'& = &\parallel v_s \parallel'_s=\sum_{n=0}^{\infty}e^{\lambda_1n}\parallel D_xf^n(v_s)\parallel\\
	& \le & \sum_{n=0}^{\infty}e^{\lambda_1n}e^{\varepsilon k}e^{-(\lambda-\varepsilon)n}\parallel v_s \parallel \\
	&= &  e^{\varepsilon k} \sum_{n=0}^{\infty}e^{-\varepsilon n} \parallel v_s \parallel.
\end{eqnarray*}
If $ \parallel v \parallel'=\parallel v_c \parallel'_c $, then 
\begin{eqnarray*}
	\parallel v \parallel'& =& \parallel v_c \parallel'_c=\sum_{n=0}^{\infty}e^{-\mu'_1n}\parallel D_xf^n(v_c)\parallel+ \sum_{n=1}^{\infty}e^{-\lambda'_1n}\parallel D_xf^{-n}(v_c)\parallel\\
	& \le & \sum_{n=0}^{\infty}e^{-\mu'_1n}e^{\varepsilon k}e^{(\mu'+\vep)n}  \parallel v_c \parallel+ \sum_{n=1}^{\infty}e^{-\lambda'_1n}e^{\varepsilon k}e^{(\lambda'+\vep)n}  \parallel v_c \parallel\\
	&\le & e^{\varepsilon k} \sum_{n=-\infty}^{\infty}e^{-\varepsilon |n|} \parallel v_c \parallel.
\end{eqnarray*}
Similarly, if $ \parallel v \parallel'=\parallel v_u \parallel'_u $, then  $ \parallel v \parallel'=\parallel v_u \parallel'_u \le e^{\varepsilon k}  \displaystyle\sum_{n=0}^{\infty}e^{-\varepsilon n} \parallel v_u \parallel$.

Denote  $ C=\displaystyle\sum_{n=-\infty}^{\infty}e^{-\varepsilon |n|} $, then 
$ \parallel v \parallel' \le C e^{\varepsilon k}\parallel v \parallel$.  Moreover, 
\begin{equation*}
\begin{aligned}
\parallel v \parallel^2 &=\langle v_s+v_c+v_u,v_s+v_c+v_u\rangle \\[2mm]
&=\parallel v_s \parallel^2 + \parallel v_c \parallel^2+\parallel v_u \parallel^2+2\langle v_s,v_c \rangle+2\langle v_s,v_u \rangle+2\langle v_s,v_u \rangle\\[2mm]
&\le 3(\parallel v_s \parallel^2 + \parallel v_c \parallel^2+\parallel v_u \parallel^2) \le 3[(\parallel v_s \parallel'_s )^2+ (\parallel v_c \parallel'_c)^2+(\parallel v_u \parallel'_u)^2] \\[2mm]
&\le 9(\parallel v \parallel')^2.
\end{aligned}
\end{equation*}
Thus we have $\parallel v \parallel \le 3 \parallel v \parallel' $. Consequently, we obtain 
\begin{equation}\label{metric relation}
\frac{1}{3}\parallel v \parallel \le  \parallel v \parallel' \le C e^{\varepsilon k}\parallel v \parallel,
\end{equation}
which means metric $ \parallel \cdot \parallel' $ is equivalent to $ \parallel \cdot \parallel $ on $ \Lambda_k $. But  metric $ \parallel \cdot \parallel' $ is not equivalent to $ \parallel \cdot \parallel $ on the whole $ \Lambda=\bigcup\limits_{k\ge 1}\Lambda_k $, as $ \lim\limits_{k\to\infty}  C e^{\varepsilon k}=+\infty$. 

By taking  small neighborhood $ U $ of   any $x\in M $ with uniform size, we can
trivialize the tangent bundle over $ U $ by identifying $ T_U M=U \times \mathbb{R}^{\dim M} $. For any point $ y\in U $
and tangent vector $ v\in T_y M $, we can use the identification $ T_U M=U \times \mathbb{R}^{\dim M} $  to translate
the vector $ v $ to a corresponding vector $ \bar{v}\in T_xM $. For $x\in \Lambda_k$,  define 
\begin{equation*}
\begin{aligned}
&\parallel \cdot \parallel'': T_{U}M\to \mathbb{R}\\
&\parallel v \parallel''_y=\parallel \bar{v} \parallel'_x
\end{aligned}
\end{equation*}
and a new splitting $ T_y M=E_y^{s''}\oplus E_y^{c''}\oplus E_y^{u''} $ by translating the splitting $ T_x M=E_x^{s}\oplus E_x^{c}\oplus E_x^{u} $.
Denoting  $E^{cs''}_y=E^{s''}_y\oplus E^{c''}_y$ and  $E^{cu''}_y=E^{c''}_y\oplus E^{u''}_y, $ we have the splittings $ T_y M=E^{cs''}_y\oplus E_y^{u''}$ and  $E_y^{cs}=  E_y^{c''}\oplus E_y^{u''} $. 
Similarly, for a small neighborhood $ V $ of $ f(x) $  and any $ z\in V $, we can define a metric $ \parallel \cdot \parallel''_z $ by $ \parallel \cdot \parallel'_{f(x)} $ and a new splitting $ T_z M=E_z^{s''}\oplus E_z^{c''}\oplus E_z^{u''} $ by translating the splitting $ T_{f(x)} M=E_{f(x)}^{s}\oplus E_{f(x)}^{c}\oplus E_{f(x)}^{u} $. 

For a splitting $F=F_1\oplus F_2$ of an Euclidean space $F$ with norm $\|\,\|$, and $\xi>0$,  we denote by $Q_{\|\,\|}(F_1,\xi)$ the cone of width $\xi$ of $F_1$ in $\|\,\|$, i.e.  the set $\{v=v_1+v_2\in F:\,v_{1}\in F_1,\,v_2\in F_2,\quad \|v_{2}\|\le \xi \|v_{1}\|\}$.

\begin{Prop}\label{Pro}
	Let $\epsilon=\vep/\alpha$.   	Take $\lambda_2=\lambda_1-\vep,\; \mu_2=\mu_1-\vep,\; \lambda_2'=\lambda_1' +\vep$ and $ \mu_2'=\mu_1'+\vep $;  $\lambda_3=\lambda_2-\vep ,\; \mu_3=\mu_2-\vep,\; \lambda_3'=\lambda_2' +\vep$ and $ \mu_3'=\mu_2'+\vep $. 
	\begin{itemize}
		\item [(i)]There exists $ \epsilon_0>0$ such that for any $x\in\Lambda_k$,  $y\in B(x, \epsilon_0 e^{-k\epsilon})$,  one has 	\begin{eqnarray*}  \Big{(}\frac{1}{\|Df^{-1}\|}-(e^{\vep}-1)\Big{)} \|v_{s}\|''\le & \parallel Df(v_s)\parallel'' &\le  e^{-\lambda_2}\parallel v_s \parallel'',\;v_s\in E_y^{s''};\\
			e^{-\lambda'_2}\parallel v_c \parallel'  \le   & \parallel Df(v_c)\parallel'' &\le  e^{\mu'_2}\parallel v_c \parallel'',\;v_c\in E_y^{c''};\\
			\Big{(} \frac{1}{\|Df\|}-(e^{\vep}-1)\Big{)} \|v_u\|''\le    &   \parallel Df^{-1}(v_u)\parallel''&\le  e^{-\mu_2}\parallel v_u \parallel'',\;v_u\in E_y^{u''}. \end{eqnarray*}
		\item [(ii)] For any small  $\xi>0$, 
		\begin{eqnarray*}   &\parallel Df(v)\parallel'' &\le  e^{-\lambda_3}\parallel v \parallel'',\;v\in Q_{\|\|_y''}(E^{s''}_y, \xi);\\
			e^{-\lambda'_3}\parallel v \parallel'  \le   & \parallel Df(v)\parallel'' &\le  e^{\mu'_3}\parallel v \parallel'',\;v \in Q_{\|\|_y''}(E^{c''}_y, \xi);\\
			&   \parallel Df^{-1}(v)\parallel''&\le  e^{-\mu_3}\parallel v \parallel'',\;v_u\in Q_{\|\|_y''}(E^{u''}_y, \xi). \end{eqnarray*}
		Moreover,  there exist $\varsigma\in (0,1)$ and $a_{\xi}>0$ such that  for any $x\in\Lambda_k$,  $y\in B(x,  a_{\xi}\epsilon_0 e^{-k\epsilon})$, \begin{eqnarray*} D_yf(Q_{\|\|_{y}''}(E^*_y, \xi))&\subset &  Q_{\|\|_{f(y)}''}(E^*_{f(y)}, \varsigma \xi),\quad *\in \{u'',\, cu''\},\\
			D_yf^{-1}(Q_{\|\|_y''}(E^{*}_y, \xi)) &\subset & Q_{\|\|''_{f^{-1}(y)}}(E^{*}_{f^{-1}(y)}, \varsigma\xi), \quad *\in \{s'', \, cs''\}.  \end{eqnarray*}
	\end{itemize}	
	
\end{Prop}

\begin{proof}  Note that $ f $ is a  $ C^{1+\alpha} $ diffeomorphism and so $ f^{-1} $ is also $C^{1+\alpha}$.   There exists  a constant   $ K >0$ such that  $$ \parallel D_xf(v)-D_yf(v) \parallel \le K|x-y|^\alpha  \parallel v \parallel\quad \text{and}\quad  \parallel D_xf^{-1}(v)-D_yf^{-1}(v) \parallel \le K|x-y|^\alpha  \parallel v \parallel.$$ 
	
	(i)  Denote by $ B(x,r) $  the ball centered at $ x $ with radius $ r $. Let
	\begin{eqnarray*}
		\epsilon_k=\min\biggl\{1,  \biggl(\frac{e^{-\lambda_2}-e^{-\lambda_1}}{3Ce^{\varepsilon(k+1)}K}\biggr)^{\frac{1}{\alpha}},
		\biggl(\frac{e^{\lambda'_2}-e^{\lambda'_1}}{3Ce^{\varepsilon(k+1)}K}\biggr)^{\frac{1}{\alpha}},   \biggl(\frac{(e^{-\mu_2}-e^{-\mu_1})}{3Ce^{\varepsilon(k+1)}K}\biggr)^{\frac{1}{\alpha}},  
		\biggl(\frac{(e^{\mu'_2}-e^{\mu'_1})}{3Ce^{\varepsilon(k+1)}K}\biggr)^{\frac{1}{\alpha}}
		\biggl\},
	\end{eqnarray*}
	then  for any $ y\in B(x,\epsilon_\kappa) $, we have 
	\begin{itemize}
		\item[(1)]$E^s$: 
		\begin{eqnarray*}
			\parallel D_yf(v_s)\parallel''_{f(y)}&=& \parallel D_yf(v_s)\parallel'_{f(x)}\\[2mm]
			&\le & \parallel D_xf(v_s)\parallel'_{f(x)}+\parallel D_yf(v_s)-D_xf(v_s)\parallel'_{f(x)}\\[2mm]
			&\le& e^{-\lambda_1}\parallel v_s \parallel'_x+Ce^{\varepsilon(k+1)}K|y-x|^{\alpha}\parallel v_s \parallel_x\\[2mm]
			&\le& (e^{-\lambda_1}+3Ce^{\varepsilon(k+1)}K|y-x|^{\alpha} )\parallel v_s \parallel'_x\\[2mm]
			&\le & e^{-\lambda_2}\parallel v_s \parallel''_y,\,\,v_s\in E_y^{s''} .
		\end{eqnarray*}
	\end{itemize}
	Besides, \begin{eqnarray*}
		\parallel D_yf(v_s)\parallel''_{f(y)}&=& \parallel D_yf(v_s)\parallel'_{f(x)}\\[2mm]
		&\ge & \parallel D_xf(v_s)\parallel'_{f(x)}-\parallel D_yf(v_s)-D_xf(v_s)\parallel'_{f(x)}\\[2mm]
		&\ge&   \frac{1}{\|Df^{-1}\|} \|v_s\|'_x-Ce^{\varepsilon(k+1)}K|y-x|^{\alpha}\parallel v_s \parallel_x\\[2mm]
		&\ge&   \frac{1}{\|Df^{-1}\|} \|v_s\|'_x-3Ce^{\varepsilon(k+1)}K|y-x|^{\alpha}\parallel v_s \parallel_x'\\[2mm]
		&\ge & \big{(}\frac{1}{\|Df^{-1}\|} - (e^{\vep}-1)\big{)} \parallel v_s \parallel_x', \,\,\,v_s\in E_y^{s''}.
	\end{eqnarray*}
	
	Similarly, we can further have that
	\begin{itemize}	
		\item[(2)]$E^c$:  \begin{equation*}  e^{-\lambda'_2}\parallel v_c \parallel''_y\le\parallel D_yf(v_c)\parallel''_{f(y)}\le e^{\mu'_2}\parallel v_c \parallel''_y,\,\,v_c\in E_y^{c'}; \end{equation*}
		\item[(3)]$E^u$:  
		\begin{equation*}  \Big{(} \frac{1}{\|Df\|}-(e^{\vep}-1)\Big{)} \|v_u\|''_y\le    \parallel D_yf^{-1}(v_u)\parallel''_{f^{-1}(y)}\le e^{-\mu_2}\parallel v_u \parallel''_y,\,\,v_u\in E_y^{u'}.\end{equation*}
	\end{itemize}

	Let $ \epsilon=\frac{\varepsilon}{\alpha} $.  At most let $\epsilon_k$ small with finite $k$,  we can  take $ \epsilon_0>0 $, which is independent of $ k $, such that $ \epsilon_{k}= \epsilon_0 e^{-\epsilon k} $.  
	
	(ii) The first part on the inequalities can be obtained by (i) by taking small $\xi$.   We are going to show the cone properties.  Using the invariance of $E^u$ and the domination property at $x$ there exists $\varsigma_1\in (0,1)$ independent of $x$ satisfying $D_xf(Q_{\|\|_x''}(E^u_x, \xi))\subset Q_{\|\|_{f(x)}''}(E^u_{f(x)},\varsigma_1\xi)$. Then for any $y\in B(x, a_\xi\epsilon_k)$, we get by  (\ref{metric relation}) and the choice of $\epsilon_k$, 
	\begin{eqnarray*}
		\|D_xf-D_yf\|_x''& = &\max_{\|v\|_x''=1}\|D_xf(v)-D_yf(v)\|_{f(x)}''\\ [2mm]
		&\leq& Ce^{(k+1)\vep}\|D_xf-D_yf\| \\ [2mm]
		& \leq& Ce^{(k+1)\vep} K (a_{\xi}\epsilon_{k})^{\alpha}\\ [2mm]
		&\leq & a_\xi^{\alpha}.
	\end{eqnarray*}
	For  small $\vep$,  by (i)   one has also \begin{eqnarray*}
		\frac{1}{2\|Df^{-1}\|}\le \min_{\|v\|_x''=1}\|D_yf(v)\|_{f(x)}''\le \max_{\|v\|_x''=1}\|D_yf(v)\|_{f(x)}''
		\le 2\|Df\|.
	\end{eqnarray*}
	Take $\varsigma\in (\varsigma_1, 1)$. It follows that for    $\|v\|_x''=1$,   the angle $\angle''( D_{y}f(v),D_{x}f(v))$ with respect to $\|\cdot\|_{f(x)}''$ is less than   $\arctan(\varsigma\xi)-\arctan(\varsigma_1\xi)$ for $a_{\xi}$ small enough. We conclude that $D_yf\left(Q_{\|\|_x''}(E^u_y, \xi\right)\subset Q_{\|\|_{f(y)}''}\left(E^u_{f(y)}, \varsigma\xi\right)$ for any $y\in B(x,a_{\xi}\epsilon_k)$.  The  cone invariance property for $E^{cu''}$, $E^{s''}$ and $E^{cs''}$ can be obtained similarily. \\
\end{proof}

\subsection{H\"older continuity of Oseledets splitting} A $k$-dimensional distribution $E$  on a smooth manifold $M$ is a family of $k$-dimensional subspaces $E(x)\subset T_xM$. A Riemannian metric on $M$ naturally induces distances in $TM$ and in the space of $k$-dimensional distributions on $TM$.   The  H\"older continuity of a distribution $E$ can be defined using these distances.   By the Whitney Embedding Theorem,  $M$ can be embedded in  an Euclidead space $\mathbb{R}^N$.          Since $M$ is compact, the Riemannian metric on $M$ is equivalent to the distance $\|x-y\|$ by the embedding for $x, y\in M$.       Let $A\subset \mathbb{R}^{N}$ be a   subspace and $v \in \mathbb{R}^{N}$  be a vector,  set $$d(v, A)=\min_{w\in A} \|v-w\|.$$
For subspaces $A$ and $B$ in $ \mathbb{R}^{N}$, define 
$$d(A, B)=\max\Big{\{}  \max_{v\in A, \|v\|=1} d(v, B),\,   \max_{v\in B, \|v\|=1} d(w, A)\Big{\}}.$$
Let $D\subset \mathbb{R}^{N} $ be a subset and let $E$ be a $k$-dimensional distribution. The distribution $E$ is said to be H\"older continuous with  H\"older exponent $\beta\in (0,1]$ and  H\"older constant $L\ge 0$ if there exists $\vep_0>0$ such that 
$$d(E(x), E(y))\le L \|x-y\|^{\beta}$$
for every $x,y\in D$ with $\|x-y\|\le \vep_0$. 

Considering $E^{s}\oplus E^{cu}$, for $x\in \Lambda_k$, let $\theta_n=\angle(E^s(f^n(x)), E^{cu}(f^n(x)))$ for $n\in\mathbb{N}$. For any unit vector $v\in (E^s(x))^{\perp}$,  denote $v=v_s+v_{cu}$ with $v_s\in E^s(x)$, $v_{cu}\in E^{cu}(x)$. It holds that $\|v_{cu}\|\ge \frac{1}{|\sin \theta_0|}\ge 1$. Then 
\begin{eqnarray*} 
	\|Df^n v\| &\ge & \|Df^n v_{cu}\| |\sin \theta_n| \\[2mm]
	&\ge & \ e^{-k\vep} e^{-n(\lambda'+\vep)} |\sin \theta_n|  \\[2mm]
	&\ge &    e^{-k\vep} e^{-n(\lambda'+\vep)}   e^{-(k+n)\vep} \\[2mm]
	&=&    e^{-2k\vep} e^{-n(\lambda'+2\vep)}.
\end{eqnarray*}
Besides, for  unit vector $v\in E^s(x)$,  \begin{eqnarray*} 
	\|Df^n v\| &\le &  e^{k\vep} e^{-n(\lambda-\vep)},
\end{eqnarray*}
and by Lemma 4.16 of \cite{BP}  for every $a>\max \{\|D_zf\|^{1+\alpha},\,\|D_zf^{-1}\|^{1+\alpha}:\, z\in M\}$,  there exists $D>1$ such that  for every $n\in\mathbb{N}$ and every $x,y\in M$ we have 
$$\|D_xf^{\pm n}-D_yf^{\pm n}\|\leqslant Da^n \|x-y\|^{\alpha}.$$
By Lemma 4.14 of \cite{BP}, we have for $x, y\in \Lambda_k$, 
\begin{eqnarray*} d(E^{s}(x), E^s(y))&\leqslant& 3(e^{2k\vep})^2\frac{e^{-(\lambda'+2\vep)}}{e^{-(\lambda-\vep)}}(D\|x-y\|^{\alpha})^{\frac{-(\lambda'+2\vep)+\lambda-\vep}{\ln a+\lambda-\vep}}\\[2mm]
	&=& 3e^{4k\vep}e^{\lambda-\lambda'-3\vep}(D\|x-y\|^{\alpha})^{\frac{\lambda-\lambda'-3\vep}{\ln a+\lambda-\vep}}
\end{eqnarray*}
Similarily,  consider  $E^{cs}\oplus E^{u}$.   For any unit vector $v\in (E^{cs}(x))^{\perp}$, it holds that 
\begin{eqnarray*} 
	\|Df^n v\| \ge    e^{-2k\vep} e^{n(\mu-2\vep)}.
\end{eqnarray*}Besides, for  unit vector $v\in E^{cs}(x)$,  \begin{eqnarray*} 
	\|Df^n v\| &\le &  e^{k\vep} e^{n(\mu'+\vep)}.
\end{eqnarray*}
We have for $x, y\in \Lambda_k$, 
\begin{eqnarray*} d(E^{cs}(x), E^{cs}(y))&\leqslant& 3(e^{2k\vep})^2\frac{e^{\mu-2\vep}}{e^{\mu'+\vep}}(D\|x-y\|^{\alpha})^{\frac{\mu-2\vep-(\mu'+\vep)}{\ln a- (\mu'+\vep)}}\\[2mm]
	&=& 3e^{4k\vep}e^{\mu-\mu'-3\vep}(D\|x-y\|^{\alpha})^{\frac{\mu-\mu'-3\vep}{\ln a-\mu'-\vep}}.
\end{eqnarray*}
Moreover, considering $f^{-1}$,  we can get that for $x, y\in \Lambda_k$, 
\begin{eqnarray*} d(E^{u}(x), E^{u}(y))&\leqslant& 3(e^{2k\vep})^2\frac{e^{-(\mu'+2\vep)}}{e^{-(\mu-\vep)}}(D\|x-y\|^{\alpha})^{\frac{-(\mu'+2\vep)+(\mu-\vep)}{\ln a+(\mu-\vep)}}\\[2mm]
	&=& 3e^{4k\vep}e^{\mu-\mu'-3\vep}(D\|x-y\|^{\alpha})^{\frac{\mu-\mu'-3\vep}{\ln a+\mu-\vep}},
\end{eqnarray*}
and 
\begin{eqnarray*} d(E^{cu}(x), E^{cu}(y))&\leqslant& 3(e^{2k\vep})^2\frac{e^{\lambda-2\vep}}{e^{\lambda'+\vep}}(D\|x-y\|^{\alpha})^{\frac{\lambda-2\vep-(\lambda'+\vep)}{\ln a-(\lambda'+\vep)}}\\[2mm]
	&=& 3e^{4k\vep}e^{\lambda-\lambda'-3\vep}(D\|x-y\|^{\alpha})^{\frac{\lambda-\lambda'-3\vep}{\ln a-\lambda'-\vep}}.
\end{eqnarray*}
Denote  $$a_1=\max\Big{\{} \frac{\lambda-\lambda'-3\vep}{\ln a+\lambda-\vep},  \,\,\frac{\mu-\mu'-3\vep}{\ln a-\mu'-\vep},\,\, \frac{\mu-\mu'-3\vep}{\ln a+\mu-\vep}, \,\, \frac{\lambda-\lambda'-3\vep}{\ln a-\lambda'-\vep} \Big{\}}.$$
Then $E^s$, $E^{cs}$, $E^{u}$, $E^{cu}$  on $\Lambda_k$ are $a_1\alpha$-H\"older continuous  with H\"older constant $b_1e^{4k\vep}$, where $b_1$ is independent of $k$ and $n$.  Since  $E^c=E^{cs}\cap E^{cu}$,  we also get that  $E^c$ on $\Lambda_k$ is $a_1\alpha$-H\"older continuous with H\"older constant $b_1e^{4k\vep}$. 

Now we consider the H\"older continuity with respect to $<, >'$. 
\begin{Prop}\label{holder}
	$E^s, E^{cu}, E^u$ and  $E^{cs} $ on $\Lambda_k$ are $a_1\alpha$-H\"older continuous  with H\"older constant $b_2e^{6k\vep}$ for some constant $b_2$, with respect to $<,>'$. 
\end{Prop}
\begin{proof}For $x, y\in \Lambda_k$, let $e_{s}\in E^{s}(x)$, $e_{cu}\in E^{cu}(x)$ and $v\in E^s(y)$ be the  unit vectors with respect to $<, >'$.   Denote by $\gamma'$ the angle between $e_s$ and $v$   with respect to $<, >'_x$ then  $$v=\cos\gamma' \cdot e_s+ \sin \gamma' \cdot  e_{cu}.$$  Denote by $\gamma$ the angle between $e_s$ and $v$   with respect to the Riemannian metric of $M$, then 
	\begin{eqnarray*} |\tan \gamma|&=& \frac{\|\sin \gamma' \sin \theta_0\cdot  e_{cu}\|}{\|\cos \gamma'\cdot  e_s\| +|\cos \theta_0|\|\sin \gamma'  \cdot e_{cu}\|}\\[2mm]
		&\geqslant& \frac{|\sin \gamma' | e^{-k\vep} C^{-1}e^{-k\vep}}{3+3},
	\end{eqnarray*}
	which implies $$\gamma'\leqslant C_1 e^{2k\vep} \gamma$$
	for some constant $C_1$.

	\vspace*{0pt}
	\begin{figure}[H]
		\begin{center}
			\includegraphics[width=0.85\textwidth]{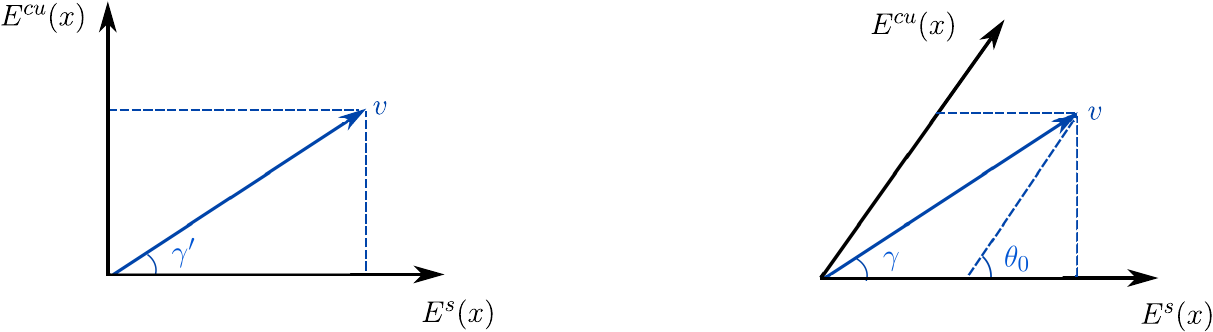}
		\end{center}
		\caption{}
	\end{figure}
	\vspace*{10pt}
	
	Note that 
	\begin{eqnarray*} \gamma&\leqslant & b_1 e^{4k\vep}\|x-y\|^{a_1\alpha}\leqslant b_1 e^{4k\vep}(3\|x-y\|')^{a_1\alpha}.
	\end{eqnarray*}
	For two subspaces  with  small angle,  any one of the distance, the angle, tangent and sine  are uniformly bounded  by any other.  Therefore, we get that $$d'(E^s(x), E^s(y))\leqslant b_2 e^{6k\vep}(\|x-y\|')^{a_1\alpha}$$
	for some constant $b_2$, where $d'$ is the distance of subspaces with respect to $<,>'$.  The  H\"older continuity of $E^{cu}, E^u$ and  $E^{cs} $ can be proved analogously.
\end{proof}

\subsection{Invariant foliations}

Note that  for  any $\eta>0$,  $ B(x, \eta)$ contains a ball centered at $x$ with radius $\frac13 \eta$  in the distance $d''$, which we denote by  $B''(x, \frac13  \eta)$.   By Proposition \ref{Pro}, it holds    uniform partial hyperbolicity   for $T_yM=E_y^{s''}\oplus  E_y^{c''}\oplus E_y^{u''}$ in $\|\cdot\|''$ for $y\in  B''(x,  \frac{1}{3}a_{\xi}\epsilon_0 e^{-k\epsilon})\subseteq  B(x,  a_{\xi}\epsilon_0 e^{-k\epsilon})$, $x\in \Lambda_k$.  
By taking  the exponential map at $x$ we can assume without loss
of generality that we are working in   $\mathbb{R}^{\mathsf{d}}$.    Let $\xi$ and $a_\xi$ be as in Proposition \ref{Pro}. For any $x\in \Lambda^*$, we can extend $f\mid_{B(x, a_{\xi}\epsilon_0 e^{-k\epsilon})}$ to a diffeomorphism $\tilde{f}_x: \mathbb{R}^{\mathsf d} \to \mathbb{R}^{\mathsf d} $ such that
\begin{itemize}
	\item $\tilde{f}_x(y)=f(y)$\quad for $y\in B(x,  a_{\xi}\epsilon_0 e^{-k\epsilon})$;\\
	\item  $\|D_y\tilde{f}_x-D_xf\|''_x\leq 2 a_{\xi}^{\alpha}$\quad for $y\in \mathbb{R}^{\mathsf d}$.                  
\end{itemize}
By taking $a_{\xi}$ smaller,  the properties  of  Proposition  \ref{Pro}  hold with respect to   $\tilde{f}_x$ for all $y\in \mathbb{R}^{\mathsf d}$.  Moreover, 
there exists $L>1$ such that 
$$\|D_y\tilde{f}_x^{\pm}\|''\le L,\quad \forall\,y\in \mathbb{R}^{\mathsf d},\,\,x\in \Lambda^*.$$

Let  $$\delta_k=((1-\varsigma)\xi(b_2e^{6k\vep}))^{-a_1\alpha}\,\,\ \text{for any}\,\,   k\in \mathbb{Z}_+.$$
For a  seqenece   $ \{a_n\}_{n\in\mathbb{Z}}  $ of positive  integers,  let   $ \{\cdots, [x_{-n}, f^{a_{-n}-1}(x_{-n})], $ $  \cdots, [x_0, f^{a_0-1}(x_0)], \cdots, [x_{n},  f^{a_n-1}(x_{n})], \cdots \}  \subset  \Lambda $   be   a sequence of orbit segments  which is a $ \{\delta_{k} \}_{k\in\mathbb{Z^+}}$ pseudo-orbit,  i.e., there exist two  sequences of positive integers $ \{s_n^{-}\}_{n\in\mathbb{Z}} $  and $\{s_n^{+}\}_{n\in\mathbb{Z}} $  such that
\begin{itemize}
	\item[(a)] $ x_n \in\Lambda^{-}_{s_n} $, \, $ f^{a_n}(x_n) \in\Lambda^{+}_{s_n} $, $ \forall n\in\mathbb{Z} $;\\
	\item[(b)] $ |s^+_n-s^{-}_{n+1}|\le 1 $,  $ \forall n\in\mathbb{Z} $;\\
	\item[(c)]$d(f^{a_n}(x_n),x_{n+1})< \delta_{s_n}$,  $\forall n\in\mathbb{Z} $.
\end{itemize}
Then by the H\"older property of Proposition \ref{holder}, 
$$d(E^*(f^{a_n}(x_{n}), E^*(x_{n+1})))<(1-\varsigma)\xi,\,\,\,*\in \{s, cs, cu, u\}.$$
Therefore, 
$$Q_{\|\|_{f^{a_n}(x_n)}''}(E^*_{f^{a_n}(x_n)}, \varsigma\xi)\subset Q_{\|\|_{x_{n+1}}''}(E^*_{x_{n+1}}, \xi),\,\,\,*\in \{s, cs, cu, u\}.$$

Denote $$\{z_i\}_{i=-\infty}^{+\infty}=\{\cdots, x_{-n},\cdots, f^{a_{-n}-1}(x_{-n}),\cdots, x_0,\cdots,f^{a_0-1}(x_0),\cdots, x_n, f^{a_n-1}(x_{n}),\cdots  \}$$ with $z_0=x_0$. 
Let $\Xi$ be the disjoint union given by $$\Xi=\ \coprod_{i=-\infty}^{+\infty} \{z_{i}\}\times \mathbb{R}^{\mathsf d}. $$   
Then $\tilde{f}=(\tilde{f}_{z_i})_{i=-\infty}^{+\infty}$  can be viewed as a map from $\Xi$ to itself by letting $ \tilde{f}(x,v)= \left(f(x),\tilde{f}_x(v)\right)$.

Note that the global splitting $\coprod_{i=-\infty}^{+\infty} \{z_{i}\}\times \mathbb{R}^{\mathsf d}=\coprod_{i=-\infty}^{+\infty} \{z_{i}\}\times (E^{s}_{z_i}\oplus \oplus E^c_{z_i}\oplus E^u_{z_i})$ is dominated with respect to $\tilde{f}$.  By \cite{HPS} $\textsection 5$ and ,  we can obtain
a family $ \{ \mathcal{Y}^{cs}_:\,x\in \Lambda^*\}$ of global invariant foliations $C^1$ submanifolds in $\mathbb{R}^{\mathsf d}$ which are $C^1$ graphs defined on $E^{cs}_x$ such that we have for all $x\in \Lambda^* $ :
\begin{eqnarray*} &&\{x\}\times \mathcal  Y^{cs}_x=\bigcap_{n=0}^{+\infty}\tilde{f}^{-n}\left(\{f^n(x)\}\times Q_{\|\|_x''}(E^{cs}_{f^n(x)}, \xi)\right),\\[2mm]
	&& \forall y\in \mathbb{R}^{\mathsf d},\ T_y\mathcal Y^{cs}_x \subset Q_{\|\|_x''}(E^{cs}_{x}, \xi).\end{eqnarray*}
In particular we get  $\tilde{f}^{\pm}(\{x\}\times \mathcal Y^{cs}_x) \subset\{f^{\pm}(x)\}\times  \mathcal Y^{cs}_{f^{\pm}(x)}$. Since we have $\tilde{f}\mid_{\{x\}\times B(x,a_{\xi}\gamma_{\kappa(x)})}=f\mid_{B(x,b_{\xi}\gamma_{\kappa(x)})}$, one concludes the proof by considering   $\mathcal W^{cs}_x=\mathcal Y_x\cap B(x,    a_{\xi}\gamma_{\kappa(x)})$ and taking much smaller $b_\xi$ than $a_{\xi}$.

\begin{Prop}\label{pro}
	Any  $  \{z_{i}\}\times \mathbb{R}^{\mathsf d}$ with $-\infty<i<+\infty$   is foliated by foliations  $ \mathcal{W}^s_{z_i} $,   $ \mathcal{W}^{cs}_{z_i} $,    $ \mathcal{W}^{cu}_{z_i} $,    $ \mathcal{W}^{c}_{z_i}$,     and   $ \mathcal{W}^u_{z_i} $ with  the following properties:     for  $*\in \{s, cs, cu, c, u\}$ and   for each $\tilde{y}=\{z_{i}\}\times \{y\} \in \{z_{i}\}\times \mathbb{R}^{\mathsf d}$, 
	\begin{itemize}
		\item[(i)] almost tangency:   the leaf $ \mathcal{W}^*_{z_i}(y)$ is $C^1$, and the tangent space  $T_{y}   \mathcal{W}^*_{z_i}(z)$ lies in a cone of radius $\epsilon$  about $E^*_{z_i}$.\\
		\item [(ii)] invariance:    $ f_{z_i}( \mathcal{W}^*_{z_i}(y)) = \mathcal{W}^*_{z_{i+1}}(f_{z_i}(y)). $\\
		\item [(iii)] exponential growth bounds:   
		
		\linespread{1.4} \selectfont		
		\begin{itemize}  
			
			\item [(a)] for any $ y_1, y_2\in \mathcal{W}^s_{z_i}(y) $,  $ d''(f_{z_i}(y_1),  f_{z_i}(y_2))\le e^{-\lambda_3} d''(y_1, y_2) $;
			
			\item [(b)]for any $ y_1, y_2\in \mathcal{W}^c_{z_i}(y) $, $ e^{-\lambda'_3} d''(y_1, y_2)\le  d''(f_{z_i}(y_1),  f_{z_i}(y_2))\le e^{\mu'_3} d''(y_1, y_2) $;
			
			\item [(c)]for any $ y_1, y_2\in \mathcal{W}^u_{z_i}(y) $,  $ d''(f_{z_i}(y_1),  f_{z_i}(y_2))\ge e^{\mu_3} d''(y_1, y_2)$.
		\end{itemize}
		
		\item [(iv)]	coherence:   $ \mathcal{W}^{s}_{z_i}$ and   $ \mathcal{W}^{c}_{z_i}$   subfoliate  $ \mathcal{W}^{cs}_{z_i}$,  and $ \mathcal{W}^{u}_{z_i}$ and   $ \mathcal{W}^{c}_{z_i}$   subfoliate  $ \mathcal{W}^{cu}_{z_i}$.\\
		
		\item [(v)] 
		regularity:  the foliations  $ \mathcal{W}^{*}$ and their tangent distributions are uniformly H$ \ddot{o} $lder continuous.

	\end{itemize}
	
\end{Prop}

For $y, z \in \mathbb{R}^{\mathsf d}$,  write
\begin{itemize}\item $\{[x, y]_{cs, u}\} =  \mathcal{W}^{cs}_{z_i}(x) \cap  \mathcal{W}^{u}_{z_i}(y)$,\\
	\item $ \{[x, y]_{s, cu}\} =  \mathcal{W}^{s}_{z_i}(x) \cap  \mathcal{W}^{cu}_{z_i}(y)$.
\end{itemize}
There  exists $b\ge 1$ such that 
\begin{eqnarray*}&d_{ \mathcal{W}^{cs}_{z_i}(x)}(x, [x, y]_{cs, u})\le b d(x, y),&\quad d_{ \mathcal{W}^{u}_{z_i}(x)}(y, [x, y]_{cs, u})\le b d(x, y),\\[2mm]
	&d_{ \mathcal{W}^{s}_{z_i}(x)}(x, [x, y]_{s, cu})\le b d(x, y),&\quad d_{ \mathcal{W}^{cu}_{z_i}(x)}(y, [x, y]_{s, cu})\le b d(x, y).
\end{eqnarray*}
By letting $\epsilon$ small, we can let $b$ close to 1 such that 
$$ \frac{b}{2} (1+e^{- \min\{\lambda_3, \,\mu_3\}+\epsilon}) <1.$$
By letting $\epsilon$ small, we can let $b$ close to 1 such that 
$$ \frac{b}{2} (1+e^{- \min\{\lambda_3, \,\mu_3\}+\theta^{-1}\epsilon}) <1.$$

By the  regularity property (v),  $\mathcal{W}_{z_i}^{*}$ is $\theta$-H\"older continuous for some $\theta\in (0,1)$  and the  constant $\tilde{C}$ can be chosen uniformly for uniformly bounded regions.

\section{Quasi-shadowing property }
\begin{proof}[\bf Proof of Theorem \ref{thA}]

	For any $z_i$,   $-\infty<i<+\infty$, we take  cylinders
	\begin{eqnarray*}C^{cu}(z_i, \gamma e^{-\theta^{-1}\tau(z_i)\epsilon})&=&\bigcup_{x\in \mathcal{W}^s_{z_i}(z_i, \gamma e^{-\theta^{-1}\tau(z_i)\epsilon})} \mathcal{W}_{z_i}^{cu}(x),\\[2mm]
		C^{cs}(z_i, \gamma e^{-\theta^{-1}\tau(z_i)\epsilon})&=&\bigcup_{x\in \mathcal{W}^u_{z_i}(z_i, \gamma e^{-\theta^{-1}\tau(z_i)\epsilon})} \mathcal{W}_{z_i}^{cs}(x),
	\end{eqnarray*}
	where $\gamma$  will be assigned   later.  Note that $$\tilde{f}(C^{cu}(z_i, \gamma e^{-\theta^{-1}\tau(z_i)\epsilon}))\subset C^{cu}(\tilde{f}(z_i),  e^{-\lambda_3}\gamma e^{-\theta^{-1}\tau(z_i)\epsilon}).$$
	Take $\delta_i= \frac{\gamma}{2} (1-e^{-\lambda_3+\theta^{-1}\epsilon}) C^{-1}e^{-(1+\theta^{-1})\tau(z_{i})\epsilon}$.  Since $d(\tilde{f}(z_i), z_{i+1})<\delta_{i+1}=  \frac{\gamma}{2} (1-e^{-\lambda_3+\theta^{-1}\epsilon}) C^{-1}e^{-(1+\theta^{-1})\tau(z_{i+1})\epsilon}$,  $$d''(\tilde{f}(z_i), z_{i+1})< \frac{1}{2} (1-e^{-\lambda_3+\theta^{-1}\epsilon}) \gamma e^{-\theta^{-1}\tau(z_{i+1})\epsilon}.$$ 
	For any $y\in \mathcal{W}^s_{f(z_i)}(e^{-\lambda_3}\gamma e^{-\theta^{-1}\tau(z_i)\epsilon})$,   we have 
	\begin{eqnarray*}d''(y, z_{i+1})&<&e^{-\lambda_3}\gamma e^{-\theta^{-1}\tau(z_i)\epsilon}+ \frac{1}{2} (1-e^{-\lambda_3+\theta^{-1}\epsilon}) \gamma e^{-\theta^{-1}\tau(z_{i+1})\epsilon},\\[2mm]
		&\le &e^{-\lambda_3+\theta^{-1}\epsilon}\gamma e^{-\theta^{-1}\tau(z_{i+1})\epsilon}+ \frac{1}{2} (1-e^{-\lambda_3+\theta^{-1}\epsilon}) \gamma e^{-\theta^{-1}\tau(z_{i+1})\epsilon}\\[2mm]
		&=&  \frac{1}{2} (1+e^{-\lambda_3+\theta^{-1}\epsilon}) \gamma e^{-\theta^{-1}\tau(z_{i+1})\epsilon}. \end{eqnarray*}
	Therefore,  $$d''_{\mathcal{W}^s(z_{i+1})}( z_{i+1}, [y, z_{i+1}]_{cu, s})< \frac{b}{2} (1+e^{-\lambda_3+\theta^{-1}\epsilon}) \gamma e^{-\theta^{-1}\tau(z_{i+1})\epsilon} \le \gamma e^{-\theta^{-1}\tau(z_{i+1})\epsilon},$$  which implies
	$$\tilde{f}(C^{cu}(z_i, \gamma e^{-\theta^{-1}\tau(z_i)\epsilon}))\subset C^{cu}(z_{i+1},  \gamma e^{-\theta^{-1}\tau(z_{i+1})\epsilon}).$$
	In the same manner, one can get that 
	$$\tilde{f}^{-1}(C^{cs}(z_i, \gamma e^{-\theta^{-1}\tau(z_i)\epsilon}))\subset C^{cs}(z_{i-1},  \gamma e^{-\theta^{-1}\tau(z_{i-1})\epsilon}).$$
	Denote 
	\begin{eqnarray*}C^{cu}_{\infty}(z_i) &=& \bigcap_{j=0}^{+\infty}\tilde{f}^j(C^{cu}(z_{i-j}, \gamma e^{-\theta^{-1}\tau(z_{i-j})\epsilon})),\\[2mm]
		C^{cs}_{\infty}(z_i) &=& \bigcap_{j=0}^{+\infty}\tilde{f}^{-j}(C^{cs}(z_{i+j}, \gamma e^{-\theta^{-1}\tau(z_{i+j})\epsilon})),
	\end{eqnarray*}
	which  consist of $cu$-leaves and $cs$-leaves, respectively.  It holds that 
	\begin{eqnarray*}\tilde{f}(C^{cu}_{\infty}(z_i)) &=& C^{cu}_{\infty}(z_{i+1}),\\[2mm]
		\tilde{f}^{-1}(C^{cs}_{\infty}(z_i)) &=& C^{cs}_{\infty}(z_{i-1}).
	\end{eqnarray*}
	Take  one $cu$-leaf $F_{z_0}\subset C^{cu}_{\infty}(z_0)$ and one $cs$-leaf $G_{z_0}\subset C^{cs}_{\infty}(z_0)$.  Denote 
	$$F_{z_i}=\tilde{f}^i(F_{z_0})\subset C^{cu}_{\infty}(z_i)\quad \text{and}\quad G_{z_i}=\tilde{f}^i(G_{z_0})\subset C^{cs}_{\infty}(z_i),\quad \forall\,i\in \mathbb{Z}.$$
	Then $F_{z_i}$ and $G_{z_i}$ are $cu$-leaf and $cs$-leaf, respectvely. 
	
	Denote $$ \{p(z_i)\}=F_{z_i}\cap \mathcal{W}_{z_i}^{s}(z_i),\quad \{q(z_i)\}=G_{z_i}\cap  \mathcal{W}_{z_i}^{u}(z_i),\quad \{t(z_i)\}=G_{z_i}\cap \mathcal{W}_{z_i}^{u}(p(z_i)).$$
	It holds that 
	$$t(z_i)\in F_{z_i}\cap G_{z_i}= \mathcal{W}_{z_i}^{c}(t(z_i)).$$
	Moreover, by the H\"older continuity of $\mathcal{W}_{z_i}^{cs}$,   
	\begin{eqnarray*}
		d_{\mathcal{W}_{z_i}^u(p(z_i))}(p(z_i), t(z_i))\le \tilde{C} (d_{\mathcal{W}_{z_i}^u(z_i)}(z_i, q(z_i)))^{\theta}\le  \tilde{C} (\gamma e^{-\theta^{-1}\tau(z_{i})\epsilon})^{\theta}= \tilde{C} \gamma^{\theta} e^{-\tau(z_{i})\epsilon}.
	\end{eqnarray*}
	And \begin{eqnarray*}
		d_{\mathcal{W}_{z_i}^s(z_i)}(z_i, p(z_i))\le    \gamma e^{-\theta^{-1}\tau(z_{i})\epsilon}.
	\end{eqnarray*}
	Thus
	\begin{eqnarray}\label{center-o}
	d''(z_i, t(z_i))\le \tilde{C} \gamma^{\theta} e^{-\tau(z_{i})\epsilon}+  \gamma e^{-\theta^{-1}\tau(z_{i})\epsilon}\le C_1 \gamma^{\theta} e^{-\tau(z_{i})\epsilon},
	\end{eqnarray}
	for some constant $C_1>1$.

	Now correspond to the  pseudo-orbit $$ \{\cdots, [x_{-n}, f^{a_{-n}-1}(x_{-n})],   \cdots, [x_0, f^{a_0-1}(x_0)], \cdots, [x_{n},  f^{a_n-1}(x_{n})], \cdots \}  \subset  \Lambda.$$  
	(I)  the distance along orbit segment $[x_{n},  f^{a_n-1}(x_{n})]$ for each $n$.   \\
	
	By the invariance of fake foliations,  $$\tilde{f}^{a_n-1}(t(x_n))=t(\tilde{f}^{a_n-1}(x_n))\in  F_{\tilde{f}^{a_n-1}(x_n)}\cap G_{\tilde{f}^{a_n-1}(x_n)}=  \mathcal{W}_{t(\tilde{f}^{a_n-1}(x_n))}^{c}(t(\tilde{f}^{a_n-1}(x_n))).$$

	By the H\"older continuity of $\mathcal{W}_{\tilde{f}^{a_n-1}(x_n)}^{cs}$,   
	\begin{eqnarray*}&&d_{\mathcal{W}_{\tilde{f}^{a_n-1}(x_n)}^{u}(p(\tilde{f}^{a_n-1}(x_n))}(p(\tilde{f}^{a_n-1}(x_n)), t(f\tilde{f}^{a_n-1}(x_n)))\\ [2mm]&\le&  \tilde{C} (d_{ \mathcal{W}^u_{\tilde{f}^{a_n-1}(x_n)}(\tilde{f}^{a_n-1}(x_n))}(f^{a_n-1}(x_n),  q(\tilde{f}^{a_n-1}(x_n)))^{\theta}\\[2mm] &\le&  \tilde{C} (\gamma e^{-\theta^{-1}\tau(\tilde{f}^{a_n-1}(x_n))\epsilon})^{\theta}\le C_1 \gamma^{\theta}e^{-\tau(\tilde{f}^{a_n-1}(x_n))\epsilon}.
	\end{eqnarray*}
	Moreover, \begin{eqnarray*}&&d_{\mathcal{W}_{\tilde{f}^{a_n}(x_n)}^{u}(p(\tilde{f}^{a_n}(x_n))}(p(\tilde{f}^{a_n}(x_n)), t(\tilde{f}^{a_n}(x_n)))\\ [2mm]&\le&  d_{\mathcal{W}_{\tilde{f}^{a_n-1}(x_n)}^{u}(p(\tilde{f}^{a_n-1}(x_n))}(p(\tilde{f}^{a_n-1}(x_n)), t(\tilde{f}^{a_n-1}(x_n))) L\\[2mm] &\le& C_1 \gamma^{\theta}e^{\epsilon}e^{-\tau(\tilde{f}^{a_n}(x_n))\epsilon}L.
	\end{eqnarray*}
	
	Note that $p(\tilde{f}^{a_n-1}(x_n))=\tilde{f}^{a_n-1}(p(x_n))$ and $t(\tilde{f}^{a_n-1}(x_n))=\tilde{f}^{a_n-1}(t(x_n))$.  It follows that 
	for $0\le j\le a_n-1 $, 
	\begin{eqnarray*}&&d_{\mathcal{W}_{\tilde{f}^j(x_{n})}^{u}(\tilde{f}^{j}(p(x_n)))}(\tilde{f}^{j}(p(x_n)),  \tilde{f}^{j}(t(x_n))) \\[2mm] &\le&  d_{\mathcal{W}_{\tilde{f}^{a_n-1}(x_n)}^{u}(p(\tilde{f}^{a_n-1}(x_n))}(p(\tilde{f}^{a_n-1}(x_n)), t(\tilde{f}^{a_n-1}(x_n)))\cdot  e^{-\mu_3(a_n-1-j)}\\[2mm]
		&\le & C_1 \gamma^{\theta}e^{-\tau(\tilde{f}^{a_n-1}(x_n))\epsilon}e^{-\mu_3(a_n-1-j)}\le C_1 \gamma^{\theta}e^{-\tau(\tilde{f}^{j}(x_n))\epsilon}.
	\end{eqnarray*}
	Besides,  we have 
	\begin{eqnarray*}&&d_{\mathcal{W}_{\tilde{f}^j(x_{n})}^{s}(\tilde{f}^{j}(x_n))}(\tilde{f}^j(x_n), \tilde{f}^{j}(p(x_n))) \\[2mm] &\le&  d_{\mathcal{W}_{x_{n}}^{s}(x_n)}(x_n, p(x_n))\cdot  e^{-\lambda_3 j}\\[2mm]
		&\le & \gamma e^{-\theta^{-1}\tau(x_n)\epsilon} e^{-\lambda_3 j}\le C_1\gamma^{\theta} e^{-\tau(\tilde{f}^j(x_n))\epsilon}.
	\end{eqnarray*}
	Thus, 
	\begin{eqnarray*}&&d''(\tilde{f}^j(x_n),  \tilde{f}^{j}(t(x_n))) \\[2mm]
		&\le& d_{\mathcal{W}_{\tilde{f}^j(x_{n})}^{s}(\tilde{f}^{j}(x_n))}(\tilde{f}^j(x_n), \tilde{f}^{j}(p(x_n)))+d_{\mathcal{W}_{\tilde{f}^j(x_{n})}^{u}(\tilde{f}^{j}(p(x_n)))}(\tilde{f}^{j}(p(x_n)),  \tilde{f}^{j}(t(x_n)))\\[2mm]
		&\le & 2C_1\gamma^{\theta} e^{-\tau(\tilde{f}^j(x_n))\epsilon}. 
	\end{eqnarray*}
	\\
	(II) the distance along center leaf. \\
	It holds that  
	\begin{eqnarray*}&&d''(t(x_{n+1}),  \tilde{f}^{a_n}(t(x_n))) \\[2mm]
		&\le&d''(t(x_{n+1}),  x_{n+1})+ d''(\tilde{f}^{a_n}(x_{n}), x_{n+1})+d''(\tilde{f}^{a_n}(x_{n}),  \tilde{f}^{a_n}(t(x_{n})).
	\end{eqnarray*}
	Note that 
	\begin{eqnarray*}
		d''(t(x_{n+1}),  x_{n+1})
		&\le & C_1 \gamma^{\theta} e^{-\tau(x_{n+1})\epsilon},
	\end{eqnarray*}
	\begin{eqnarray*}d''(f^{a_n}(x_{n}), x_{n+1})
		&\le &  \gamma C_0e^{-(1+\theta^{-1})\tau(x_{n+1})\epsilon},
	\end{eqnarray*}
	and \begin{eqnarray*}d''(\tilde{f}^{a_n}(x_{n}),  \tilde{f}^{a_n}(t(x_{n}))
		&\le & d''(\tilde{f}^{a_n-1}(x_{n}),  \tilde{f}^{a_n-1}(t(x_{n}))) L\\[2mm]
		&\le &  2C_1\gamma^{\theta} e^{-\tau(\tilde{f}^{a_n-1}(x_n))\epsilon}L\\[2mm]
		&\le &  2C_1\gamma^{\theta}e^{\vep} e^{-\tau(\tilde{f}^{a_n}(x_n))\epsilon}L\\[2mm]
		&\le &  2C_1\gamma^{\theta}e^{2\vep} e^{-\tau(x_{n+1})\epsilon}L.
	\end{eqnarray*}
	Therefore, 
	\begin{eqnarray*} d''(t(x_{n+1}),  \tilde{f}^{a_n}(t(x_n))) \le C_2 \gamma^{\theta} e^{-\tau(x_{n+1})\epsilon},
	\end{eqnarray*}
	for some constant $C_2>2LC_1>1$.  For any $\eta>0$,  let $\gamma=(\frac{\epsilon_0\eta}{3C_2})^{\frac{1}{\theta}} $, we can get the estimate in Theorem \ref{thA}. 
	Moreover, there exists small  $\eta_0>0$ such that for any $\gamma\in (0,(\frac{\epsilon_0\eta_0}{C_2})^{\frac{1}{\theta}} )$, the above estimates are done in  the corresponding $B''(z,  \frac{1}{3}a_{\xi}\epsilon_0 e^{-\tau(z)\epsilon})$ thus in  $B(z, a_{\xi}\epsilon_0 e^{-\tau(z)\epsilon})$ for any $z=z_i$, where $\tilde{f}=f$.

\end{proof}

\section{the growth of  quasi-periodic orbits }
Firstly, we introduce the following Katok.s definition of metric entropy (\cite{MR0573822}).
For every $0<\delta<1 $ and  an ergodic $ f$-invariant   measure  $ m $, denote   $
d_n^f(x,y)=\underset{0\le i\le n-1}{\max}d(f^i(x),f^i(y))$ and denote by $ N_f(n,\gamma,\delta) $ the minimal number of $ \gamma $-balls in the $ d_n^f $-metric which cover the set of measure more than or equal to $ 1-\delta $. Then 

\begin{equation*}
h_m(f)=\lim_{\gamma\to 0}\varliminf_{n\to \infty}\frac{\ln N_f(n,\gamma,\delta)}{n}=\lim_{\gamma\to 0}\varlimsup_{n\to \infty}\frac{\ln N_f(n,\gamma,\delta)}{n}.
\end{equation*}

Recall   the set of quasi-periodic points:  $ P'_n(f)=\{x:\;  f^n(x)\in\mathcal{W}^c(x)\} $.

\begin{proof}[\bf Proof of Theorem \ref{thC}]  
	Fix a positive integer $ k $. Then for any small positive number $ \gamma $, any positive number $ l $ and any positive integer $ n $, we shall construct a finite set $ K_n=K_n(\gamma,l) $ satifying the following properties (for some $ n $ the set $ K_n(\gamma,l) $ may be empty):
	\begin{itemize}
		\item[(i)]$K_n\subset  \Lambda_k $;
		\item[(ii)]if $ x $, $ y\in K_n $ and $ x\not =y $, then $ d^f_n(x,y)>l^{-1} $;
		\item[(iii)]for every $ x\in K_n $ there exists a number $ m(x):n\le m(x) \le (1+\gamma )n $ such that $ f^{m(x)}(x)\in\Lambda_k $ and 
		\begin{equation*}
		d(x,f^{m(x)}(x))\le \beta(k,(3l)^{-1});
		\end{equation*}
		\item[(iv)]for every $ \gamma>0 $, 	denote by $ \#K_n(\gamma,l) $  the cardinality of $ K_n(\gamma,l) $,
		\begin{equation*}
		\lim_{l\to\infty}\varliminf_{n\to\infty}\frac{\ln \# \;K_n(\gamma,l)}{n}\ge h_m(f).
		\end{equation*}
		
	\end{itemize}

	To construct the set $K_n  $, we choose a finite partition $ \xi $ satisfying 
	\begin{equation*}
	diam \; \xi <\beta(k,1/(3l))
	\end{equation*}
	and
	\begin{equation*}
	\xi >\{\Lambda_k,M\setminus \Lambda_k\},
	\end{equation*}
	which means that every element of $ \xi $ either belongs to the set $ \Lambda_k $ or is disjoint from this set.
	
	Let $ \Lambda_{k,n} =\{x\in\Lambda_k:\exists\; m\in [n,(1+\gamma)n]   $ such that the points $ x $ and $ f^m(x) $ belong to the same element of $ \xi \}$. Define $ K_n $ as the subset of the set $ \Lambda_{k,n} $ with the maximal cardinality satisfying the property (ii). The properties (i) and (iii) are true by definition. 
	
	By the property (ii), we know that the union of $ l^{-1} $-balls in the  $ d^f_n $-metric around points of $ K_n $ covers the set $ \Lambda_{k,n} $.  Then we have 
	
	\begin{equation*}
	\#K_n(\gamma,l)\ge   N_f(n,l^{-1},1-m(\Lambda_{k,n})) .
	\end{equation*}

	\begin{Lemma}\label{lem3}	
		\begin{equation*}
		\lim_{n\to \infty} m(\Lambda_{k,n})=m(\Lambda_k).
		\end{equation*}
		
	\end{Lemma}	
	
	\begin{proof}
		We fix an elemnent $ A\in\xi $ which is a subset of the set $ \Lambda_{k} $ and set
		\begin{equation*}
		A_{n,\gamma}=\{x\in A: \,  \sum_{k=0}^{n-1}\;\chi_A (f^k(x))<nm(A)(1+\frac{\gamma}{3}),\;\sum_{k=0}^{[n(1+\gamma)]}\chi_A (f^k(x))>nm(A)(1+\frac{2\gamma}{3})\},
		\end{equation*}
		where $ \chi_D $ is a characteristic function of a set $ D $.
		
		Obviously $ A_{n,\gamma}\subset \Lambda_{k,n}\cap A $.  By the ergodic theorem we have $ m(A\setminus A_{n,\gamma})\to 0 $, as $ n\to\infty $. Applying these arguments to every element $ A\in\xi $ belonging to $ \Lambda_{k} $ we obtain that 
		\begin{equation*}\lim_{n\to \infty} m(\Lambda_{k,n})=m(\Lambda_k).
		\end{equation*}
		
	\end{proof}
	
	By Lemma \ref{lem3}, we can get that for every $ \delta>1-m(\Lambda_{k}) $, we have
	\begin{equation}\label{eq24}
	\varliminf_{n\to \infty}\frac{\ln \#K_n}{n}\ge \varliminf_{n\to \infty}\frac{\ln N_f(n,l^{-1},\delta)}{n}.
	\end{equation}
	
	By Katok.s definition of metric entropy, we know that the property (iv) is true.
	
	From Theorem \ref{thB}, for every point $ x\in K_n $ we can 
	find a quasi-periodic point $ z(x) $ with  period $ m(x) $ satisfying $ d(f^i(x),f^i(z(x)))<(3l)^{-1}, \;i=0,1,\cdots,m(x)-1$.
	If $ x,y\in K_n $ and $ x\not =y $ then
	
	\begin{equation*}\begin{aligned}
	&\mathop{\max}_{0\le i\le n-1}d(f^i(z(x)),f^i(z(y)))\\[2mm]
	\ge&\mathop{\max}_{0\le i\le n-1} d(f^i(x),f^i(y))-\mathop{\max}_{0\le i\le n-1}d(f^i(x),f^i(z(x)))-\mathop{\max}_{0\le i\le n-1}d(f^i(y),f^i(z(y)))\\[2mm]
	\ge&\mathop{\max}_{0\le i\le n-1} d(f^i(x),f^i(y))-\mathop{\max}_{0\le i\le m(x)-1}d(f^i(x),f^i(z(x)))-\mathop{\max}_{0\le i\le m(x)-1}d(f^i(y),f^i(z(y)))\\[2mm]
	\ge&l^{-1}-(3l)^{-1}-(3l)^{-1}\\[2mm]
	\ge&(3l)^{-1}.
	\end{aligned}
	\end{equation*}
	So that  $ z(x)$ and  $z(y) $ are different. Consequently, we get
	
	\begin{equation*}
	\sum_{m=n}^{[(1+\gamma)n]}\#P'_{m,1}(f, (3l)^{-1})\ge \# K_n(\gamma,n)
	\end{equation*}
	and
	\begin{equation*}
	\mathop{\max}\limits_{n\le m\le(1+\gamma)n }\#P'_{m,1}(f, (3l)^{-1})\ge \frac{\#K_n(\gamma,l)}{\gamma n+1}.
	\end{equation*}
	
	Thus, we can find a sequence of integers $ \{m_n \}$ where $ n\le m_n \le (1+\gamma )n\ $ such that
	\begin{equation*}
	\begin{aligned}
	\varliminf_{n\to \infty} \frac{\ln \#P'_{m_n, 1}(f, (3l)^{-1})}{m_n}&\ge \varliminf_{n\to \infty}\frac{n}{m_n}\cdot \frac{\ln \#K_n(\gamma,l)-\ln(\gamma n+1)}{n}\\
	&\ge \frac{1}{1+\gamma } \varliminf_{n\to \infty} \frac{\ln \# K_n(\gamma,l)}{n}.
	\end{aligned}
	\end{equation*}

	By  \eqref{eq24}, we have 
	
	\begin{equation*}
	\begin{aligned}
	\varliminf_{l\to\infty}\varliminf_{n\to \infty} \frac{\ln \#P'_{m_n,1}(f, (3l)^{-1})}{m_n} 
	&=\varliminf_{l\to\infty}\varliminf_{n\to \infty} \frac{\ln \#P'_{m_n,1}(f, (3l)^{-1})}{m_n}\\
	&\ge  \frac{1}{1+\gamma } \cdot \lim_{l\to\infty}  \varliminf_{n\to \infty} \frac{\ln \#K_n(\gamma,l)}{n}\\
	&\ge \frac{1}{1+\gamma } \cdot  \varliminf_{l\to \infty} \varliminf_{n\to \infty}  \frac{\ln N_f(n,l^{-1},\delta)}{n}\\
	&=\frac{1}{1+\gamma } h_m(f).
	\end{aligned}
	\end{equation*}
	
	Let $ \gamma\to 0 $, we can get the result. 
	
\end{proof}

\bibliographystyle{plain}
\bibliography{ref}

\end{document}